\documentclass[a4paper,11pt]{amsart}

\usepackage{pdfsync}
\usepackage{setspace}

\usepackage[english]{babel}
\usepackage[latin1]{inputenc}
\usepackage{amsmath,amsfonts,amssymb,amscd,amsthm,amsbsy,epsf}

\usepackage{amssymb, latexsym, graphics, showidx}
\usepackage{color}
\usepackage{mdwlist}
\usepackage{enumitem}

\numberwithin{equation}{section}
\newtheorem{thm}{Theorem}[section]
\newtheorem{de}[thm]{Definition}
\newtheorem{rem}[thm]{Remark}
\newtheorem{cor}[thm]{Corollary}
\newtheorem{prop}[thm]{Proposition}
\newtheorem{lem}[thm]{Lemma}

\renewcommand{\dim}{\noindent\textbf{Proof.} }
\newcommand{\dims}{\noindent\textbf{Proof} }
\newcommand{\finedim}{{\unskip\nobreak\hfil\penalty50
   \hskip2em\hbox{}\nobreak\hfil\mbox{$\Box$ \qquad}
   \parfillskip=0pt \finalhyphendemerits=0\par\medskip}}
\newcommand{\R}{\mathbb{R}}

\newcommand{\Om}{\Omega}

\newcommand{\lam}{\lambda}
\newcommand{\al}{\alpha}

\newcommand{\p}{\partial}
\newcommand{\s}{\sigma}
\newcommand{\ep}{\epsilon}

\newcommand{\om}{\omega}
\newcommand{\ochi}{\overline{\chi}_-}
\newcommand{\chiep}{\chi^\ep_+}
\newcommand{\fs}{\overline{f}}
\newcommand{\us}{\overline{u}}

\newcommand{\beq}{\begin{equation}}
\newcommand{\eeq}{\end{equation}}
\newcommand{\beqs}{\begin{equation*}}
\newcommand{\eeqs}{\end{equation*}}
\newcommand{\beqa}{\begin{eqnarray}}
\newcommand{\eeqa}{\end{eqnarray}}
\newcommand{\beqas}{\begin{eqnarray*}}
\newcommand{\eeqas}{\end{eqnarray*}}

\title[]{Stochastic homogenization of a porous-medium type equation}

\author{Stefania Patrizi}
 \thanks{The author has been supported by the
NSF Grant DMS-2155156 " Nonlinear PDE methods in the study of interphases"}

\address[Stefania Patrizi]{
Department of Mathematics,
University of Texas at Austin,
2515 Speedway Stop C1200,
Austin, Texas 78712-1202, USA}

\email{spatrizi@math.utexas.edu}

\begin{document}
\maketitle

\begin{abstract}
We consider the homogenization problem for the stochastic porous-medium type equation $\p_{t}
u^\epsilon =\Delta f\left(T\left(\frac{x}{\ep}\right)\om,u^\ep\right)$, with a well-prepared initial datum, 
where  $f(T(y)\om,u)$ is  a  stationary process, increasing in $u$,  on a given probability space $(\Om, \mathcal{F}, \mu)$ endowed with an ergodic dynamical system 
$\{T(y)\,:\,y\in\R^N\}$. Differently from the previous literature \cite{afs,fs}, here we do not assume $\Om$ compact. 
We first show that the weak solution $u^\ep$ satisfies a kinetic formulation of the equation, then we exploit the theory of 
 "stochastically two-scale convergence in the mean" developed in \cite{bmw} to show convergence of the kinetic solution to the kinetic solution of an homogenized problem of the form
 $\p_{t}
\overline{u} - \Delta \overline{f}(\overline{u})=0$. 
 The homogenization result for the  weak solutions then follows.
\end{abstract}

\section{Introduction}

In this paper we study the behavior as $\ep\rightarrow 0$ of the weak solution $u^\ep(\cdot,\cdot,\om)$ of the porous-medium type equation:  for any $\om\in\Om$
\begin{equation}\label{uepeq}
\begin{cases}
\p_{t}
u^\epsilon =\Delta f\left(T\left(\frac{x}{\ep}\right)\om,u^\ep\right)&\text{in}\quad \R^+\times\R^N\\
u^\epsilon(0,x)=u_0\left(x,T\left(\frac{x}{\ep}\right)\om\right)& \text{on}\quad \R^N
\end{cases}
\end{equation}
where  $f(T(y)\om,u)$ is  a  stationary process, increasing in $u$,  on a given probability space $(\Om, \mathcal{F}, \mu)$ endowed with an ergodic dynamical system 
$\{T(y)\,:\,y\in\R^N\}$. Our model example is 
\beq\label{modelf} f(\om,u)=a(\om)u|u|^{\gamma(\om)}+b(\om)\eeq
with   $\gamma, a,b$ bounded and $\gamma(\om)\geq\gamma_0>0$, $a(\om)\geq a_0>0$ for a.e. $\om\in\Om$.
The initial datum is assumed to be "well-prepared", i.e., of the form $u_0(x,\om)=g(\om,\varphi(x))$,   for some $\varphi\in L^\infty(\R^N)\cap  L^1(\R^N)$, with  $g(\om,\cdot)=f(\om,\cdot)^{-1}$.

The homogenization  problem for porous-medium type equations of the form $\p_{t}u^\ep-\Delta\left(f\left(\frac{x}{\ep},u^\ep\right)\right)$
 in the case in which $f(\cdot,u)$ belongs to an ergodic algebra with mean value, has been studied in \cite{afs,fs}. An  example is when  $f(\cdot,u)$ is almost periodic. 
In this situation it can be proven that the algebra can be identified with the space $C(\Om)$ for some compact set $\Om$ endowed with a Borel probability measure and a continuous ergodic dynamical system. The 
 porous-medium equation then can be  written in the form \eqref{uepeq}. 
 Homogenization is then proven by establishing  the existence of multiscale limit Young
measures  associated with the family of solutions $\{u^\ep\}$, and then by showing that such measures are actually  Dirac masses  concentrated at the solution of an homogenized  porous-medium type limit problem.  
In this setting, the compactification of $\R^N$ provided by the algebras with mean value  plays a fundamental  role.

There is an extensive literature about the homogenization of non-linear first and second order PDE's in periodic and almost periodic settings, starting from the seminal paper \cite{lpv}. In more recent years,  
there has been a resurgence of interest in the homogenization problems in the more general  setting of random stationary ergodic media, see e.g. \cite{ls} and the references therein.
The main difficulty when passing from the periodic or almost periodic setting to the stochastic one is given by the lack of compactness of the probability space, see \cite{ls}.

Motivated by these  results, the goal of this paper is to solve the homogenization problem \eqref{uepeq} removing the compactness assumption on $\Om$. 
 We use a different approach than \cite{afs,fs}, based on the kinetic formulation for the equation \eqref{uepeq}.
The notion of kinetic solutions  for hyperbolic homogeneous conservation laws has been introduced by Lions, Perthame and Tadmor  \cite{lpt},  and then extended by Chen and Perthame  \cite{cp} to parabolic laws which include, as special case, the homogeneous porous-medium equation 
$\p_t u-\Delta f(u)=0.$
In \cite{d}, Dalibard defines a notion of kinetic solutions for heterogeneous parabolic conservation laws of type $\p_t u+\text{div}(A(x,u(x))-\Delta u=0,$ where $A(x,u)$ is a given  flux, periodic in the space variable $x$. The kinetic formulation  is then used  to prove periodic homogenization. In this paper, following this idea, we  derive a kinetic formulation for the Cauchy problem associated to  the heterogeneous porous-medium type equation of the form
\begin{equation*}
\p_{t}
u - \Delta f(x,u)=0\quad\text{in }\R^+\times\R^N.\end{equation*}
The corresponding kinetic equation involves an additional variable and its distributional solution is a discontinuous function. Nevertheless, the advantage of using it  is that the kinetic equation is linear. We can therefore apply the theory of the "stochastically two-scale convergence in the mean" developed in \cite{bmw} by Bourgeat  Mikelic and S. Wright.
These theory is an extension to the stochastic setting of the notion of two-scale convergence previously introduced by Allaire \cite{a} in the context of periodic functions. 

We show the existence of  the kinetic solution of \eqref{uepeq} and we prove its convergence, as $\ep\to 0$, to the kinetic solution of an homogenized porous-medium equation of the form 
 \begin{equation*}
\p_{t}
\overline{u} - \Delta \overline{f}(\overline{u})=0\quad\text{in } \R^+\times\R^N
\end{equation*}
with initial condition $\overline{u}(0,x)= \int_\Om u_0\left(x,\om\right)d\mu$, where $\overline{f}$ is defined by the formula
 \beqs v=   \int_\Om g(\omega,\overline{f}(v))d\mu, \quad v\in\R\eeqs and $g(\omega,\cdot)=f(\omega,\cdot)^{-1}$. The proof of the convergence result is inspired by the one of the contraction  property of kinetic solutions of parabolic conservation laws, see \cite{cp}. Going back from the kinetic  to the weak solution of  \eqref{uepeq}, we prove  the convergence of   the weak solution $u^\epsilon$ of  \eqref{uepeq} to the weak solution of the homogenized problem. 
 
In order to apply the results in  \cite{bmw}, we need to require  the space $L^2(\Om)$ to be separable.  In the almost periodic case, or more in general in the case of ergodic algebras with mean value,  the corresponding $L^2(\Om)$ is not separable. However in \cite{Casado} the authors have been able to extend the theory of two-scale convergence to this setting,
we will recall their result in Section \ref{ergodicalgebras}. Therefore,  the strategy we adopt here to prove homogenization works  also in the context  of ergodic algebras,  providing a different proof of the analogous  results in 
  \cite{afs,fs} for the case $\Om$ compact. 
  
  \subsection{Organization of the paper}
  The paper is organized as follows. 
 In Section \ref{ansatzsec} we derive the ansatz for $u^\epsilon$. The main convergence result,  Theorem \ref{mainthm},  is stated in Section \ref{mainresultsec}.
 In Section \ref{prelminarsec} we collect some preliminary results concerning the ergodic theory and the stochastically two-scale convergence in the mean that will be used later in the paper. The kinetic formulation for \eqref{uepeq} is then derived in Section \ref{kineticformulation}. Finally, Section \ref{proofmain} is devoted to the proof of Theorem \ref{mainthm}.


\section{Main result}
\subsection{The ansatz}\label{ansatzsec}
In order to identify the limit of the solutions $u^\ep(t,x)$ of \eqref{uepeq} as $\ep\to 0$, following the classical idea of the two-scale expansion (see \cite{blp} for a general presentation of this theory), 
we are looking for an  ansatz of the form $U^0\left(t,x,\frac{x}{\ep}\right)$, that is a function  such that  $u^\ep(t,x)-U^0\left(t,x,\frac{x}{\ep}\right)$ converges to 0 as $\ep\to 0$ in some norm. 
To simplify the presentation, we suppose  that we are in periodic setting, i.e., that $u^\ep$ is solution of 
\beq\label{periodicuepeq}\p_{t}
u^\epsilon = \Delta f\left(\frac{x}{\ep},u^\ep\right)\eeq
 with $f(y,u)$  periodic in $y$ and increasing in $u$. 
We consider the following  two-scale expansion for $u^\ep$:
$$u^\ep(t,x)=U^0\left(t,x,\frac{x}{\ep}\right)+\ep U^1\left(t,x,\frac{x}{\ep}\right)+\ep^2U^2\left(t,x,\frac{x}{\ep}\right)+\dots$$
where $U^0(t,x,y),\, U^1(t,x,y)$ and $U^2(t,x,y)$ are periodic in $y$ functions .
Putting this expression in  \eqref{periodicuepeq} and making a Taylor expansion of $f\left(\frac{x}{\ep},\cdot\right)$ around $U^0$, we get
\beqs \p_tU^0-\Delta_x\left[f\left(\frac{x}{\ep},U^0\right)+\p_u f\left(\frac{x}{\ep},U^0\right)(\ep U^1+\ep^2U^2)+\frac{1}{2}\p^2_{uu}f\left(\frac{x}{\ep},U^0\right)\ep^2(U^1)^2\right]+O(\ep)=
0,\eeqs
where $O(\ep)$ contains all the terms multiplied by a power of $\ep$ greater or equal than 1.
Now, identifying the non-positive powers of $\ep$, we   derive an  equation for $U^0$.
The equation corresponding to  $\ep^{-2}$  is the the following 
\beq\label{celleqper} \Delta_y(f(y,U^0(t,x,y)))=0,\eeq where we denote $y=x/\ep$ and, as usual in deriving an ansatz for homegenization problems, we assume $x$ and $y$ to be independent. By the Liouville Theorem  all periodic in $y$ solutions $f(y,U^0(t,x,y))$ 
of \eqref{celleqper} in $\R^N$  are constant in $y$, that is 
$$f(y,U^0(t,x,y))=p(t,x)$$
for any function $p(t,x)$ independent of $y$. We infer that $U^0(t,x,y)$ is of the following form
$$U^0(t,x,y)=g(y,p(t,x))$$
where $$g(y,\cdot)=f^{-1}(y,\cdot).$$
The equation corresponding to the power $\ep^0$ is 
\beq\label{ep0equ}\begin{split} \p_t&U^0-\Delta_x(f(y,U^0))-\text{div}_{y}\nabla_x(\p_uf(y,U^0)U^1)-\text{div}_{x}\nabla_y(\p_uf(y,U^0)U^1)\\&
-\Delta_{y}(\p_uf(y,U^0)U^2)-\frac{1}{2}\Delta_{y}(\p^2_{uu}f(y,U^0)(U^1)^2)=0.
\end{split}\eeq
Assuming that the functions are smooth, we get
\beqs\begin{split} 
\text{div}_{x}\nabla_y(\p_uf(y,U^0)U^1)&=\sum_{i=1}^N\partial_{x_i}\partial_{y_i}(\p_uf(y,U^0)U^1)=\sum_{i=1}^N\partial_{y_i}\partial_{x_i}(\p_uf(y,U^0)U^1)
\\&=\text{div}_{y}\nabla_x(\p_uf(y,U^0)U^1).
\end{split}\eeqs
Thus, averaging  \eqref{ep0equ} with respect to $y$ and using that, by periodicity, 
$$\int_{\mathcal{T}^N}\text{div}_{y}\nabla_x(\p_uf(y,U^0)U^1)\,dy=\int_{\mathcal{T}^N}\Delta_{y}(\p_uf(y,U^0)U^2)\,dy
=\int_{\mathcal{T}^N}\Delta_{y}(\p^2_{uu}f(y,U^0)(U^1)^2)\,dy=0,
$$
with $\mathcal{T}^N$ the $N$-dimensional torus, yields the evolution equation 
$$\p_t \overline{u}-\Delta p(t,x)=0,$$
where 
 \beq\label{defubaransatz}\overline{u}(t,x):= \int_{\mathcal{T}^N} U^0(t,x,y)dy=\int_{\mathcal{T}^N}g(y,p(t,x))dy,\eeq
 and we have used that  $f(y,U^0(t,x,y))=p(t,x)$.
The  computations above suggest to define the function 
 $\overline{f}$  implicitly in the following way, for any given $u\in\R$,
\beq\label{deffbaransatz}u=   \int_{\mathcal{T}^N} g(y,\overline{f}(u))dy.\eeq 
Then,  from \eqref{defubaransatz} and \eqref{deffbaransatz} we have that $p(t,x)=\overline{f}(\overline{u}(t,x))$ . We conclude that if  
 $\overline{u}$ is the solution of 
\begin{equation*}
\begin{cases}
\p_{t}
\overline{u} - \Delta \overline{f}(\overline{u})=0&\text{in}\quad \R^+\times\R^N\\
u(0,x)=\int_{\mathcal{T}^N} u_0\left(x,y\right)dy& \text{on}\quad \R^N
\end{cases}
\end{equation*} with $\overline{f}$ defined by \eqref{deffbaransatz},  our guess for $U^0$ is the following 
$$U^0\left(t,x,\frac{x}{\ep}\right)=g\left(\frac{x}{\ep},\overline{f}(\overline{u}(t,x))\right).$$


\subsection{Assumptions and main result.} \label{mainresultsec} Let us now introduce the mathematical assumptions we make, we refer
to Section \ref{prelminarsec}  for  the main definitions and some preliminary results. 
Throughout this paper we assume that $(\Om, \mathcal{F}, \mu)$ is a probability space with $L^2(\Om)$ separable,  and  that $\{T(y)\,:\,y\in\R^N\}$ is an ergodic $N$-dimensional dynamical system on $(\Om, \mathcal{F}, \mu)$.  Moreover, we assume that $f:\Om\times\R\rightarrow\R$ is a measurable function  satisfying,  for a.e. $\om\in\Om$:
\renewcommand{\labelenumi}{({\bf H\arabic{enumi}})}
\begin{enumerate}
\item \label{1}  $f(\om,\cdot)$ is striclty increasing and locally Lipschitz  continuous, uniformly in  $\om$. Moreover, $\lim_{u\rightarrow\pm \infty}f(\om,u)=\pm \infty$, uniformly in $\om$;
\item \label{2} $f(T(\cdot)\om,u)$ is continuous and  $f(\cdot,u)\in L^\infty(\Om)$ for all $u\in \R$. 
\item \label{4}Let $g(\om,\cdot):=f(\om,\cdot)^{-1}$, then for all $p\in\R$ and $\frac{\partial g}{\partial p}(T(\cdot)\om,\cdot)\in L_{loc}^{1}(  \R^N\times\R)$ uniformly in $\om$.\end{enumerate} 
We assume that the initial data $u_0$ is "well-prepared", that is of the form 
\begin{enumerate}[resume]
\item \label{3} $u_0(x,\om)=g(\om,\varphi(x))$,   for some $\varphi\in L^\infty(\R^N)\cap  L^1(\R^N)$ and $g(\om,\varphi)-g(\om,0)\in L^1(\Om;L^1(\R^N))$.
\end{enumerate} 

\medskip

Observe that (H\ref{1})-(H\ref{4}) are satisfied by functions of the form \eqref{modelf}
 with $\gamma, a,b$ bounded stochastic variables, with  $\gamma (T(\cdot))$, $a (T(\cdot))$ and $b (T(\cdot))$ continuous, and $\gamma(\om)\geq\gamma_0>0$, $a(\om)\geq a_0>0$ for a.e. $\om\in\Om$. Moreover, in this case  (H4) is satisfied for any $\varphi\in L^\infty(\R^N)\cap  L^1(\R^N)$.

Under the assumptions (H\ref{1})-(H\ref{3}),  even though $\partial_u f$ can be 0 at some point,  \eqref{uepeq}  still belongs to the  "non-degenerate"  class according to the classification of
 \cite{c}.  Well-posedness  of \eqref{uepeq} in the homogeneous case, i.e.  when the coefficients do not depend explicitly on $(t,x)$, was established by Carrillo in \cite{c}.
 The results of \cite{c} have been extended by Frid and Silva  in \cite{fs} to the case in which $f$ explicitly depends on $x$ and satisfies  (H\ref{1})-(H\ref{2}),   these results in particular 
 guarantee that  for any fixed $\ep>0$ and a.e. $\om\in\Om$ there exists a unique unique weak solution $u^\ep(t,x,\om)$  of \eqref{uepeq} (see Definition \ref{weaksolded}  for  the definition of the weak solution to \eqref{uequ}). 
 Assumptions  (H\ref{4}) and  (H\ref{3}) are needed in order to make sense to the kinetic formulation of \eqref{uepeq} given in Section \ref{kinetichetero}.

A $L^1$ well-posedness theory for the homogeneous anisotropic case, that is for  non-diagonal viscosity matrices was established by Chen and Perthame \cite{cp} and succesiveley extended to the non-homogeneous case in \cite{ck}.

Let $\overline{g}$ be the function defined by 
\beq\label{gbareq}\overline{g}(p)=  \int_\Om g(\omega,p)d\mu, \quad p\in\R.\eeq Since $g(\om,\cdot)$ is strictly increasing,  $\overline{g}$ has inverse $\overline{f}:=\overline{g}^{-1}$  implicitly defined  by the equation 
\beq\label{fbareq}v=   \int_\Om g(\omega,\overline{f}(v))d\mu,\quad v\in\R.\eeq
\begin{lem}\label{fbarproperties}
Let $\overline{f}:\R\to\R$ be defined by \eqref{fbareq}. Then $\overline{f}$ is strictly increasing and locally Lipschitz continuous in $\R$.
\end{lem}
For the  proof of Lemma \ref{fbarproperties} we refer to the proof of  Lemma 6.1 in \cite{fs}.  
 By the results of \cite{c} and Lemma \ref{fbarproperties},   there exists a unique weak solution of the Cauchy problem 
\begin{equation}\label{uequ}
\begin{cases}
\p_{t}
\overline{u} - \Delta \overline{f}(\overline{u})=0&\text{in}\quad \R^+\times\R^N\\
\overline{u}(0,x)= \int_\Om u_0\left(x,\om\right)d\mu& \text{on}\quad \R^N.
\end{cases}
\end{equation}
 We are now ready to state our main result. 

\begin{thm}\label{mainthm} Let $\overline{u}$ be the unique weak solution of 
\eqref{uequ},
where $ \overline{f}$ is defined by \eqref{fbareq}. Set
$$U(t,x,\om):=g(\om,\overline{f}(\overline{u}(t,x))),$$
then  as $\ep\rightarrow 0$,   we have 
\beq\label{mainthmresult}\int_\Om\left\|u^\ep(t,x,\om)-U\left(t,x,T\left(\frac{x}{\ep}\right)\om\right)\right\|_{L^1_{loc}(\R^+\times\R^N)}d\mu \to 0,\eeq and
$\int_\Om u^\ep d\mu\rightarrow\overline{u}$ in the weak star topology of $L^\infty(\R^+\times\R^N)$.
\end{thm}

\section{Preliminary results}\label{prelminarsec}
\subsection{Ergodic theory}\label{ergodictheory}
Let us recall some basic facts about the ergodic theory that will be needed in the next sections, we refer to the book  \cite{jko} for a more complete presentation.

Let $(\Om, \mathcal{F}, \mu)$ be a probability space.
\begin{de}An $N$-dimensional dynamical system on $(\Om, \mathcal{F}, \mu)$ is a family of maps
$T(y):\Om\rightarrow\Om$, $y\in\R^N$, which satisfies the following conditions:
\begin{itemize} 
\item[(i)] {\bf (Group property)} $T(0)=I$, where $I$ is the identity map on $\Om$, and $T(y+z)=T(y)T(z)$, $\forall y,z\in\R^N$;
\item[(ii)] {\bf (Invariance)} The maps $T(y):\Om\rightarrow\Om$ are measurable and $\mu(T(y)E)=\mu(E)$, $\forall y\in \R^N,\, \forall E\in  \mathcal{F}$;
\item[(iii)]{\bf (Measurability)} Given any $F\in  \mathcal{F}$ the set $\{(y,\omega)\in \R^N\times\Om\,:\, T(y)\omega\in  F\}\subset \R^N\times\Om$ is measurable 
with respect to the $\sigma$-algebra product 
$\mathcal{L}_N\otimes  \mathcal{F},$ where $\mathcal{L}_N$ is the $\sigma$-algebra of the Lebesgue measurable sets of $\R^N$. 
\end{itemize}
\end{de}


\begin{de}[Ergodic $N$-dimensional dynamical system] A $\mathcal{F}$-measurable function $f:\Om\rightarrow\R$ is called invariant if $f(T(y)\om)=f(\omega)$ 
$\mu$-almost everywhere in $\Om$, for all $y\in\R^N$. A dynamical system is said to be ergodic if every invariant function is $\mu$-equivalent 
to a constant in $\Om$.  
\end{de}
\begin{de}[Stationary process] A stochastic process $\tilde{F}:\R^N\times\Om\rightarrow\R$ is called stationary if
$$\tilde{F}(y+y',\om)=\tilde{F}(y,T(y')\om)\quad \text{for all }y,y'\in\R^N\text{ and a.e. }\om\in\Om.$$ If  the $N$-dimensional dynamical system $\{T(y):y\in \R^N\}$ is ergodic, then 
$\tilde{F}$ is said stationary ergodic.
\end{de}
\begin{rem}
It is easily checked that a stochastic process $\tilde{F}:\R^N\times\Om\rightarrow\R$ is stationary if and only if there exists a stochastic variable $F:\Om\rightarrow\R$ such that
$$\tilde{F}(y,\om)=F(T(y)\om).$$
\end{rem}
Given the probability space $(\Om, \mathcal{F}, \mu)$,  as usual, for $1\le p<+\infty$, let us denote by $L^p(\Om)=L^p(\Om,\mu)$ be the space of the equivalent classes of  measurable functions $g:\Om\to\R$ such that $|g|^p$ is $\mu$-integrable on $\Om$, and by $L^\infty(\Om)=L^\infty(\Om,\mu)$ the space of $\mu$-essentially bounded measurable functions. 
Let $T(y)$, $y\in\R^N$, be an  $N$-dimensional dynamical system on   $(\Om, \mathcal{F}, \mu)$. If $g\in L^p(\Om)$, then almost all its realizations $g(T(y)\om)$ belong to $L^p_{loc}(\R^N)$. 
Moreover, $T(y)$ induces a group $\{U(y)\,:\,y\in\R^N\}$ of unitary operators on $L^2(\Om)$ defined by 
$$(U(y)h)(\om)=h(T(y)\om),\quad y\in\R^N,\,\om\in\Om,\,h\in L^2(\Om)$$ which turns out to be strongly continuous in $L^2(\Om)$.

Let $D_1,...,D_N$ denote the infinitesimal generators of the group with $\mathcal{D}_1,...,\mathcal{D}_N$ their respective domains in  $L^2(\Om)$, i.e., for $h\in\mathcal{D}_i$
$$(D_ih)(\om):=\lim_{y_i\neq 0,\,y_i\rightarrow 0\atop y_j=0,\,j\neq i} \frac{h(T(y)\om)-h(\om)}{y_i},\quad i=1,\ldots,N$$ in the sense of convergence  in $L^2(\Om)$. 
Then,
for  $h\in\mathcal{D}_i$,  for a.e. $\om\in\Om$, the  realization $h(T(y)\om)$ possesses a weak derivative $\p_{y_i}(h(T(y)\om))\in L^2_{loc}(\R^N)$ and the following equality holds
\beq\label{Dh=ph}(D_ih)(T(y)\om)=\p_{y_i}(h(T(y)\om))\quad\text{for a.e. }y\in\R^N.\eeq  The unitary of the group
 $\{U(y)\,:\,y\in\R^N\}$ implies that the operators $D_i$ are skew-symmetric, i.e., for $h,\, g\in \mathcal{D}_i$ we have 
 $$\int_\Om D_i h g d\mu=-\int_\Om hD_i  g d\mu\quad i=1,\ldots,N.$$
Define $\mathcal{D}(\Om)=\cap_{i=1}^N \mathcal{D}_i$  and 
\beq\label{Dinfinity}D^\infty(\Om)=\{h\in L^\infty(\Om)\cap \mathcal{D}(\Om)\,:\, D^\al h\in L^\infty(\Om)\cap \mathcal{D}(\Om),\text{ for all multi-indeces }\al\}.\eeq
For a function $h\in L^2(\Om)$,  the  stochastic weak derivative $D^\al f$ of $f$ is the linear functional on $\mathcal{D}^\infty(\Om)$ defined by
$$(D^\al f)\varphi=(-1)^{|\al|}\int_\Om fD^\al\varphi d\mu,\quad \varphi\in\mathcal{D}^\infty(\Om).$$
The following result is proven in \cite{bmw}.
\begin{lem}[\cite{bmw}, Lemma 2.3]\label{lemgrad0const}Assume  the dynamical system $\{T(y):y\in \R^N\}$ to be ergodic and $L^2(\Om)$ separable. Let $h\in L^2(\Om)$ such that $D_ih=0$ for any $i=1,\ldots,N$, then $h$ is $\mu$-equivalent to a constant in $\Om$.
\end{lem}

Using Lemma \ref{lemgrad0const}, we can prove the following Liouville type result that will be needed  in Section \ref{proofmain}.
\begin{lem}\label{liouvillelemma}Assume  the dynamical system $\{T(y):y\in \R^N\}$ to be ergodic and $L^2(\Om)$ separable.  Let  $h\in L^\infty(\Om)\cap L^2(\Om)$ such that $\Delta h=0$, then $h$ is $\mu$-equivalent to a constant in $\Om$.
\end{lem}
\dim
Let us introduce a smooth approximation of $h$. A classical way to do it consists in introducing 
an even function $K$ such that 
$$K\in C^\infty_0(\R^N),\quad \int_{\R^N}K(z)dz=1,\quad K\ge 0,$$
and set
$$h^\delta(\om)=\int_{\R^N}K_\delta(z)h(T(z)\om)dz,$$ 
where $K_\delta(z)=\delta^{-N}K(\delta^{-1}z)$.  It turns out that $h^\delta\in D^\infty(\Om)$, $U(y)h^\delta$ is infinitely differentiable as a function of $y\in\R^N$ and  $$\lim_{\delta\rightarrow 0}\|h^\delta-h\|_{ L^2(\Om)}=0,$$ see  \cite{jko}. Moreover $h^\delta$ satisfies 
$$\int_\Om \Delta \varphi(\om)h^\delta(\om)d\mu=-\sum_{i=1}^n\int_\Om  D_i \varphi(\om)D_i h^\delta(\om)d\mu=0$$ for any $\varphi\in D^\infty(\Om)$. Lemma \ref{lemgrad0const} then implies that $D_i h^\delta(\om)$ is equivalent to a constant in $\Om$. In particular, for a.e. $\om\in\Om$ and any $y\in\R^N$, $\p_{y_i}(h^\delta(T(y)\om))=(D_i h^\delta)(T(y)\om)$ is constant. Since in addition $h^\delta\in L^\infty(\Om)$, we infer that $\p_{y_i}(h^\delta(T(y)\om))=0$ for a.e. $\om\in\Om$, i.e., $h^\delta(T(y)\om)=h^\delta(\om)$ for a.e. $\om\in\Om$ and every $y\in\R^N$. The ergodicity of the dynamical system $\{T(y):y\in \R^N\}$ then implies that $h^\delta$ is equivalent to a constant in $\Om$.
Passing to the limit  as $\delta\rightarrow 0$ we conclude   that $h$ is equivalent to a constant in $\Om$.
\finedim


\subsection{Stochastically two-scale convergence in the mean}
Following an idea of Nguetseng \cite{n}, Allaire \cite{a} defined the notion of two-scale convergence in the periodic setting. Bounded sequence in $L^2(Q)$, where $Q$ is a given domain, are proven to be relatively compact with respect to this  type of convergence. 
The notion of two scale convergence is useful  for the  homogenization of partial differential equation with periodically oscillating coefficients.  In order to threat equations with random coefficients,  in \cite{bmw} Bourgeat et al. extend  this theory from the periodic to the stochastic setting, introducing the concept of  "stochastically two-scale" convergence in the mean. They prove the following:

\begin{thm}[\cite{bmw}, Theorem 3.4]\label{thm3.4} Let $(\Om, \mathcal{F}, \mu)$ be a probability space such that $L^2(\Om)$ is separable and let $\{T(y): y\in \R^N\}$ be a N-dynamical system. Let $Q$ be an open set of $\R^N$ and let $\{w^\ep\}$ be a bounded sequence in $L^2(Q\times\Om)$ . Then there exists a subsequence, still denoted by $\{w^\ep\}$, and a function 
$w_0\in  L^2(Q\times \Om)$ such that 
$$\lim_{\ep\rightarrow 0}\int_{Q\times\Om} w^\ep(x,\om)\psi\left(x,T\left(\frac{x}{\ep}\right)\om\right)dxd\mu=\int_{Q\times\Om}w_0(x,\om)\psi(x,\om)dxd\mu$$ for any  function $\psi$ such that $\psi(x,T(x)\om)$ defines an element of $L^2(Q\times \Om)$. Such a sequence $\{w^\ep\}$ is said to "stochastically two-scale" converge in the mean to $w_0(x,y)$.
\end{thm}
\begin{rem}
Not for every element  $\psi\in L^2(Q\times \Om)$,  $\psi(x,T(x)\om)$ defines an element of $L^2(Q\times \Om)$ but  if for example $\psi(x,\om)=g(x)h(\om)$ 
with
$g\in L^2(Q)$ and $h\in L^2(\Om)$, then $(x,\om)\to \psi(x,T(x)\om)$ belongs to   $L^2(Q\times \Om)$, see \cite{bmw}.
\end{rem}
We will need to apply the previous result to sequence of functions belonging to $L^\infty(Q\times \Om)$.
With a minor modification of the proof given in \cite{bmw}, the concept of "stochastically two-scale" convergence can be extended to  $L^\infty$ functions.

\begin{prop}\label{propstocconv} Under the same assumptions of Theorem \ref{thm3.4}, 
let $\{w^\ep\}$ be a bounded sequence in $L^\infty(Q\times\Om)$. Then there exists a subsequence, still denoted by $\{w^\ep\}$, and a function 
$w_0\in  L^\infty(Q\times \Om)$ such that 
$$\lim_{\ep\rightarrow 0}\int_{Q\times\Om} w^\ep(x,\om)\psi\left(x,T\left(\frac{x}{\ep}\right)\om\right)dxd\mu=\int_{Q\times\Om}w_0(x,\om)\psi(x,\om)dxd\mu$$ for any  function $\psi$ such that $\psi(x,T(x)\om)$ defines an element of $L^1(Q\times \Om)$.
\end{prop}

\subsection{Ergodic algebras with mean value.}\label{ergodicalgebras}
In this subsection we recall the Bohr compactification of the set of almost periodic function on $\R^N$ and more in general of ergodic algebras with mean value. 
We will then 
present the two-scale convergence result proven \cite{Casado}.

The set of almost periodic functions on $\R^N$, here denoted by $AP(\R^N)$,  is a linear subspace of the space of bounded uniformly continuous functions on $\R^N$, that forms an algebra with mean value. This means that  $AP(\R^N)$ satisfies the following conditions:
\begin{itemize}
\item[a)] if $f,\,g\in AP(\R^N)$ then $fg\in AP(\R^N)$;
\item[b)]  $AP(\R^N)$ with the uniform convergence topology is complete;
\item[c)] the constant functions belong to $AP(\R^N)$;
\item[d)] $AP(\R^N)$ is invariant under the translations $\tau_y:\R^N\to\R^N$, $\tau_y(x)=x+y$, $y\in\R^N$, that is if $f\in AP(\R^N)$ then $f(\tau_y(\cdot))\in AP(\R^N)$; 
\item[e)]  any element $f\in AP(\R^N)$ possesses a mean value, that is there exists a number $M(f)$ such that 
$$M(f)=\lim_{\ep\to0}\frac{1}{|A|}\int_A f\left(\frac x\ep\right)\,dx$$
for any Lebesgue measurable bounded set $A\subset \R^N$.
\end{itemize}
The Besicovitch space of order $p$, with $1\le p<+\infty$, denoted by $B^p$ is defined as the closure of $ AP(\R^N)$ for the seminorm
$$[f]_p:=\left(M(|f|^p)\right)^\frac{1}{p}.$$
The Besicovitch space of order $\infty$,  $B^\infty$, is defined by
$$B^\infty:=\{f\in B^1\,|\,[f]_\infty:=\sup_{p\ge 1}[f]_p<+\infty\}.$$
The spaces $B^p$ are seminormed spaces. The quotient of $B^p$ with the kernel of $[\cdot]_p$, denoted by $\mathcal{B}^p$, is a normed space.
It is well known, see 
\cite{ds} and \cite{af}, that  there exists a compact space $\mathbb{G}^N$,  called Bohr compactification of $AP(\R^N)$, and an isometric isomorphism $i:AP(\R^N)\to C(\mathbb{G}^N)$ identifying 
$AP(\R^N)$ with the algebra $C(\mathbb{G}^N)$ of the continuous functions on $\mathbb{G}^N$. If $\mathfrak{m}$ is the Haar measure on $\mathbb{G}^N$ normalized to be a probability measure, then 
$$\int_{\mathbb{G}^N}f\,d\mathfrak{m}=M(f).$$
Moreover,  the translations $\tau_y$ induce a family of homeomorphisms $T(y):\mathbb{G}^N\to \mathbb{G}^N$, $y\in\R^N$, which is an ergodic continuous $N$-dimensional dynamical system
on $(\mathbb{G}^N, \mathcal{G},\mathfrak{m})$, with $ \mathcal{G}$ the $\sigma$-algebra of Borel sets on $\mathbb{G}^N$. Finally, the space $\mathcal{B}^p$, $1\le p\le+\infty$,
is isometrically isomorphic to $L^p(\mathbb{G}^N,\mathfrak{m})$.

More in general, if $\mathcal{A}$ is an algebra with mean value, i.e., satisfies (a)-(e) and $\mathcal{B}^p$ are the generalized Besicovitch spaces associated to 
$\mathcal{A}$, then
\begin{thm}[\cite{afs}, Theorem 4.1]\label{ambrosias}The following holds:
\begin{itemize}
\item[i)]  There exist a compact space $K$ and an isometric isomorphism $i$ identifying $\mathcal{A}$ with the algebra $C(K)$ of continuous functions on $K$.
\item[ii)] The translations $\tau_y$ induce a family of homeomorphisms $T(y):K\to K$, $y\in\R^N$, which is a continuous $N$-dimensional dynamical system.
\item[iii)]The mean value on  $\mathcal{A}$ extends to a Radon probability measure $\mathfrak{m}$ on $K$ defined by, for $f\in\mathcal{A}$,
$$\int_{K}f\,d\mathfrak{m}=M(f),$$
 which is invariant by the group of homeomorphisms $T(y)$.
 \item[iv)] For  $1\le p\le+\infty$, the Besicovitch space  $\mathcal{B}^p$ is isometrically isomorphic to $L^p(K,\mathfrak{m})$.
\end{itemize}
\end{thm}

An algebra with mean value  is called ergodic if any function  belonging to
 $\mathcal{B}^2$  and  invariant with respect to $\tau_y$ is equivalent (in $\mathcal{B}^2$) to a constant.  In this case the $N$-dynamical system given in Theorem \ref{ambrosias} is ergodic. 
Lemma \ref{lemgrad0const} for ergodic algebras with mean value is proven in \cite{afs} (see Lemma 3.2). 

Thanks to the Theorem \ref{ambrosias}, equation
$$\partial_t u=\Delta f\left(\frac{x}{\ep},u\right)$$ when $f(\cdot, u)$ is almost periodic or more in general belongs to a linear algebra with mean value, can be written as in \eqref{uepeq} by setting
$$\Om:=K,\quad \mu:=\mathfrak{m}$$ and 
$T(y)$, $y\in\R^N$, the $N$-dynamical system induced on  $K$ by the translations $\tau_y$ of $\R^N$.
 However, we cannot apply Theorem \ref{thm3.4}  and Proposition \ref{propstocconv}  in this framework as the Besicovitch spaces are in general not separable.
 In  \cite{Casado} the authors were able to overcome this difficulty and extend the theory of two-scale convergence to generalized  Besicovitch spaces, see Definition 4.1 there for the notion of two-scale convergence in this setting. 
   \begin{thm}[\cite{Casado}, Theorem 4.10]
 Let $Q$ be an open set of $\R^N$ and let $\{w^\ep\}$ be a bounded sequence in $L^p(Q\times\Om)$, $1<p\leq \infty$. Then there exists a subsequence, still denoted by $\{w^\ep\}$, and a function 
$w_0\in  L^p(Q;\mathcal{B}^p)$ such that $\{w^\ep\}$ two-scale converges to $w_0$. 
\end{thm}


\section{The Kinetic formulation}\label{kineticformulation}
In this section we derive a kinetic formulation for the heterogeneous porous-medium equation  \eqref{uepeq} that will be used in the proof of Theorem \ref{mainthm}. Let us start by recalling some classical results.
\subsection{Homogeneous porous-medium type equations}
The notion of kinetic solutions  for hyperbolic homogeneous conservations laws has been introduced by Lions, Perthame and Tadmor  \cite{lpt},  and then extended by Chen and Perthame  \cite{cp} to parabolic laws  that include, as special case, the homogeneous isotropic porous-medium equation 
\beq\label{poroueqhom}\p_t u-\Delta f(u)=0\eeq
(see also \cite{p}).
The formulation can be derived from the Kruzkhov's inequalities. Formally, if we multiply the equation by $S'(u)$, where $S$ is a $C^2$ function, we find the following equation 
$$\p_t(S(u))-\text{div}(f'(u)\nabla (S(u)))=-S''(u)f'(u)|\nabla u|^2.$$
The choice $S(u)=(u-v)_+$  for any $v\in\R$ as a limiting case of $C^2$ functions, gives the entropy inequalities
$$\p_t(u(t,x)-v)_+-\Delta(f(u(t,x))-f(v))_+=-m$$ with
$$m(t,x,v):=\delta_{u}(v)f'(u)|\nabla u|^2$$ where $\delta_{u}(v)$ is the Dirac masse at $v=u$. 
Differentiating with respect to  $v$ the previous equation, we find 

\beq\label{paraconseqnoxkin}\p_t \chi_+-f'(v)\Delta \chi_+=\frac{\p}{\p v} m\eeq
where
$$\chi_+(t,x,v):={\bf 1}_{\{u(t,x)>v\}},$$  and ${\bf 1}_A$ denotes the indicator function of the set $A$. 
The same kind of equation holds for $\chi_-(t,x,v):={\bf 1}_{\{u(t,x)<v\}}$:
$$\p_t \chi_--f'(v)\Delta\chi_-=-\frac{\p}{\p v} m.$$ The function that occurs in the kinetic formulation is the function $\chi:\R^2\rightarrow\{-1,0,1\}$ defined by
\beqs 
\chi(v,u)=\begin{cases}
1&\text{for}\quad 0<v<u,\\
-1&\text{for}\quad u<v<0,\\
0&\text{otherwise}.
\end{cases}
\eeqs
Since \eqref{paraconseqnoxkin} is linear, we see, at least formally, that the function 
\beqs \chi(t,x,v):=\chi(v,u(t,x))={\bf 1}_{v>0}\chi_+(t,x,v)-{\bf 1}_{v<0}\chi_-(t,x,v)\eeqs is still solution of \eqref{paraconseqnoxkin}.
We are now ready to give the definition of kinetic solution for \eqref{poroueqhom} with initial condition 
\beq\label{initialhom}u(0,x)=u_0(x)\in\L^1(\R^N).\eeq

\begin{de}[Definition 2.2, \cite{cp}]\label{kindefhom} A kinetic solution of \eqref{poroueqhom}, \eqref{initialhom}  is a function $u\in L^\infty([0,\infty);L^1(\R^N))$ such that 
\begin{itemize}
\item[(i)] For any $\xi\in C_c^\infty(\R)$, $\xi(u)f'(u)|\nabla u|^2\in L^1([0,\infty)\times\R^N)$;
\item[(ii)] $\chi(t,x,v)=\chi(v,u(t,x))$ satisfies  \eqref{paraconseqnoxkin}   in the sense of distributions in $[0,+\infty)\times \R^N\times\R$, with initial data $\chi(0,x,v)=\chi(v,u_0(x))$;
\item[(iii)] If $n$ is the positive measure on $\R$ defined by 
$$\int_{\R}\xi(v) dn(v):=\int_{\R^+\times\R^N\times\R}\xi(v)m(t,x,v)\,dt\,dx$$ for $\xi\in C_c(\R)$, 
then there exists $\eta\in L^\infty(\R)$ such that $\eta\rightarrow 0$ as $|v|\rightarrow\infty$ and
$$n\le \eta$$
in the sense of distributions in $\R$
\end{itemize}
\end{de}
Remark that a new real-valued variable, denoted by $v$, has been added in the kinetic formulation in order to make sense to the derivative $\frac{\p}{\p v} m$.
In \cite{cp} is shown that the notion of kinetic solution is well posed in $L^1$:
\begin{thm}[\cite{cp}, Theorem 1.2] Assume $f\in W^{1,\infty}_{loc}(\R^N)$, $f'\geq 0$ in $\R$ and $u_0\in L^1(\R^N)$. Then, there exists a unique kinetic solution 
$u\in C([0,\infty);L^1(\R^N))$ for the Cauchy problem \eqref{poroueqhom}, \eqref{initialhom}.
\end{thm}
If the initial data belongs to $L^1(\R^N)\cap L^\infty(\R^N)$, then the notion of kinetic solution is equivalent to the one of entropy solution, see \cite{cp} for the definition of entropy solution and the proof of the equivalence result. However the former is more general than the latter since is well defined in the $L^1$-setting. 
The $L^1$-stability of $L^\infty$-entropy solutions of porous-medium type equations was already proven by Carrillo \cite{c}. In the paper is also shown that if $f$ has continuous inverse, then any weak solution of \eqref{poroueqhom}, \eqref{initialhom} is also an entropy (and then a kinetic) solution.


\subsection{Heterogeneous porous-medium type equations}\label{kinetichetero}

In \cite{d}, Dalibard defines a notion of kinetic solutions for parabolic conservation laws of type $\p_t u+\text{div}(A(x,u(x))-\Delta u=0.$  In the homogeneous case, constants are stationary solutions, while they no longer play a special role in the context of heterogeneous conservation laws. Starting from this remark, already pointed out in \cite{ap}, she derives a  definition of kinetic solution  taking in the entropy inequalities   $S(u)=(u-v(x))_+$  with $v$ stationary solution.

In this paper, following this idea, we  obtain a kinetic formulation for the heterogeneous porous-medium type equations of the form
\begin{equation}\label{porousequ}
\p_{t}
u - \Delta f(x,u)=0\quad\text{in } \R^+\times\R^N,
\end{equation} with initial data
\begin{equation}\label{porousequinicond}
u(0,x)=u_0(x)\quad \text{on } \R^N.
\end{equation}
Throughout  this section  on $f$ we  assume: 

\renewcommand{\labelenumi}{({\bf f\arabic{enumi}})}
\begin{enumerate}
\item \label{f1} $f(x,\cdot)$ is strictly increasing and locally Lipschitz  continuous uniformly in $x$. Moreover, $\lim_{u\rightarrow\pm \infty}f(x,u)=\pm \infty$, uniformly in $x$;
\item \label{f2} $f(\cdot,u)$ is continuous and bounded for all $u\in\R$. 
\suspend{enumerate}

 Let us recall the notion of weak solution of the Cauchy problem \eqref{porousequ}, \eqref{porousequinicond}. 
\begin{de}\label{weaksolded} A function $u\in L^\infty(\R^+\times\R^N)$ is said to be a weak solution of \eqref{porousequ}, \eqref{porousequinicond} if the following holds:
\begin{itemize}
\item[(i)]  $f(x,u(t,x))\in L_{loc}^2(\R^+; H^1_{loc}(\R^N))$;
\item[(ii)] for any $\phi\in C^\infty_c([0,\infty)\times\R^N)$, we have 
\beqs \int_{\R^+\times\R^N}[u\phi_t-\nabla f(x,u)\cdot \nabla\phi]\,dt\,dx+\int_{\R^N}u_0\phi(0,x)\,dx=0.\eeqs
\end{itemize}
\end{de}
The existence of a unique  weak solution of  \eqref{porousequ}, \eqref{porousequinicond} is proven in \cite{fs} under the assumptions (f\ref{f1}) and (f\ref{f2}), see also \cite{afs}.  Adapting the  techniques of  \cite{c} in order to handle the explicit dependence on $x$ of $f$,   the authors also show that weak solutions of  \eqref{porousequ} satisfy an $L^1$-stability property. These results are recalled in the following:

\begin{thm}[\cite{fs}, Theorem 5.1]\label{afsthm2}  Assume   (f\ref{f1}), (f\ref{f2}) 
and $u_0\in  L^\infty(\R^N)$, then we have the following: 
\begin{itemize}
\item[(i)] There exists a unique weak solution $u\in L^\infty(\R^+\times\R^N)\cap C([0,+\infty);L^1_{loc}(\R^N))$ of \eqref{porousequ}, \eqref{porousequinicond}. 
\item[(ii)] If $u_1,\, u_2 $ are weak solutions of  \eqref{porousequ} with initial data respectively $u_{01},\,u_{02}\in L^\infty(\R^N)$, then for all $\phi\in C^\infty_c([0,\infty)\times\R^N)$, $\phi\ge 0$, we have 
\beqs\begin{split}& \int_{\R^+\times\R^N}[(u_1(t,x)-u_2(t,x))_+\phi_t+(f(x,u_1(t,x))-f(x,u_2(t,x)))_+\Delta\phi]\,dt\,dx\\&+\int_{\R^N}(u_{01}(x)-u_{02}(x))_+\phi(0,x)\,dx\geq 0.
\end{split}\eeqs
\end{itemize}
Moreover (ii) holds true also with the positive part replaced by the negative part. 
\end{thm}
If $u_2$ is stationary, inequality (ii) of Theorem \ref{afsthm2} is a consequence of the next lemma we are going to state and which is proven in \cite{fs} (see the proof of Theorem 5.1). 
The lemma is a central tool in our analysis in order to get a kinetic formulation for the Cauchy problem  \eqref{porousequ}, \eqref{porousequinicond}. Let us first introduce some notation.
Let  $H_\sigma:\R\rightarrow\R$  be the approximation of the Heaviside function given by
\beq\label{heaviside} H_\sigma(s):=\begin{cases} 1, & \text{for}\quad s>\sigma,\\
\frac{s}{\sigma}, & \text{for}\quad 0<s\le \sigma,\\
0,& s\leq 0.
\end{cases}
\eeq
Moreover, for $k\in\R$, let us  define 
\beqs B_{\sigma}^k(x,\lam):=\int_k^\lam  H_\sigma(f(x,r)-f(x,k))dr.\eeqs

\begin{lem}\label{afsthm} Assume   (f\ref{f1}), (f\ref{f2})  and
let $u_1,\,u_2$ be weak solutions of  \eqref{porousequ} with initial data respectively $u_{01},\,u_{02}\in L^\infty(\R^N)$. Assume that $u_2=u_{02}$ is a stationary solution. Then, for all $\phi\in C^\infty_c([0,+\infty)\times\R^N)$ we have 
\beq\label{entropyeqthm} \begin{split}&-\int_{\R^+\times\R^N}B_\sigma^{u_2(x)}(x,u_1(t,x))\phi_t\,dt\,dx-
\int_{\R^N}B_\sigma^{u_2(x)}(x,u_{01}(x))\phi(0,x))\,dx
\\&+\int_{\R^+\times\R^N}H_\sigma(f(x,u_1(t,x))-f(x,u_2(x)))\nabla [f(x,u_1(t,x))-f(x,u_2(x))]\cdot\nabla \phi\, dt\,dx\\&
=-\int_{\R^+\times\R^N}|\nabla [f(x,u_1(t,x))-f(x,u_2(x))]|^2H_\sigma'(f(x,u_1(t,x))-f(x,u_2(x)))\phi\, dt\,dx.
\end{split}\eeq
\end{lem}
Let $g(x,\cdot):=f^{-1}(x,\cdot)$, then by   (f\ref{f1}) and  (f\ref{f2}), one can easily show that 
$g:\R^N\times\R\to\R$ is continuous, $g(\cdot,p)\in L^\infty(\R^N)$  for any fixed $p\in\R$, and $\lim_{p\to\pm\infty}g(x,p)=\pm\infty$ uniformly in $x$. 
Moreover  $p=f(x,g(x,p))\in L_{loc}^2(\R^+; H^1_{loc}(\R^N))$. 
Thus, for any $p\in\R$, the function 
\beq\label{correctprel}v(x,p):=g(x,p),\eeq 
 is a stationary solution of \eqref{porousequ}.  We can therefore apply 
identity \eqref{entropyeqthm}  with $u_1(t,x)=u(t,x)$ the weak solution of \eqref{porousequ}, \eqref{porousequinicond}, and $u_2(x)=v(x,p)$. The entropy formulation for 
\eqref{porousequ}, given in the next proposition, is obtained by 
passing to the limit as $\sigma\rightarrow0$.
In order to make sense to the limit of the right-hand side of \eqref{entropyeqthm}, we have to consider  $p\in\R$ as  a new real-valued variable. 
\begin{prop}Assume  (f\ref{f1}),  (f\ref{f2})  and $u_0\in L^\infty(\R^N)$. Let $u$ be  the weak solution of \eqref{porousequ}, \eqref{porousequinicond}  and   
$v(x,p)$ be defined as in 
\eqref{correctprel}, then we have 
\beq\label{u-v+eq} \p_t( u(t,x)-v(x,p))_+-\Delta(f(x,u)-p)_+=-m, \eeq
in the sense of distributions in $[0,\infty)\times\R^N\times\R$,
where $$m(t,x,p)=|\nabla f(x,u)|^2\delta_{f(x,u)}(p)$$
is a nonnegative measure on $\R^+\times\R^N\times\R$.
\end{prop}
\proof 
We have to show that for any $\psi\in C_c^\infty([0,+\infty)\times\R^N\times\R)$,
\begin{equation}\begin{split}\label{lemmakinecform} 
&\int_{\R^+\times\R^N\times\R}\{-(u(t,x)-v(x,p))_+\psi_t+\nabla [(f(x,u(t,x))-p)_+]\cdot\nabla \psi\} \,dt\,dx\,dp\\&-\int_{\R^N\times\R}(u_0(x)-v(x,p))_+\psi(0,x,p)\,dx\,dp
\\&=-\int_{\R^+\times\R^N}|\nabla f(x,u(t,x))|^2\psi(t,x ,f(x,u))\,dt\,dx.
\end{split}
\end{equation}
It suffices to show \eqref{lemmakinecform}  for $\psi(t,x,\xi)=\phi(t,x)\xi(p)$ with  $\phi\in C_c^\infty([0,+\infty)\times\R^N)$ and $\xi\in C_c^\infty(\R)$. 
Applying  \eqref{entropyeqthm}  with $u_1(t,x)=u(t,x)$, $u_2(x)=v(x,p)$ and then integrating in $p$, we find

\beq \begin{split}\label{lemprelkin1}&-\int_{\R}\xi(p)\int_{\R^+\times\R^N}B_\sigma^{v(x,p)}(x,u(t,x))\phi_t\,dt\,dx\,dp-\int_{\R}\xi(p)\int_{\R^N}B_\sigma^{v(x,p)}(x,u_0(x))\phi(0,x)\,dx\,dp
\\&+\int_{\R}\xi(p)\int_{\R^+\times\R^N}H_\sigma(f(x,u(t,x))-f(x,v(x,p)))\nabla [f(x,u(t,x))-f(x,v(x,p))]\cdot\nabla \phi \,dt\,dx\,dp\\&
=-\int_{\R}\xi(p)\int_{\R^+\times\R^N}|\nabla [f(x,u(t,x))-f(x,v(x,p))]|^2H_\sigma'(f(x,u(t,x))-f(x,v(x,p)))\phi \,dt\,dx\,dp.
\end{split}\eeq

Remind that $f(x,v(x,p))=p$. Moreover 

\beqs H_\sigma'(f(x,u)-f(x,v))=\begin{cases} 
\frac{1}{\sigma}, & \text{for}\quad f(x,u)-\sigma<p<f(x,u),\\
0, & \text{for}\quad p< f(x,u)-\sigma\text{ or }p>f(x,u).\\
\end{cases}
\eeqs

Hence the right-hand side of \eqref{lemprelkin1} is equal to the following quantity
\beqs -\int_{\R^+\times\R^N}|\nabla f(x,u(t,x))|^2\phi(t,x)\int_{f(x,u)-\sigma}^{f(x,u)}\frac{\xi(p)}{\sigma}dpdtdx.\eeqs

Passing to the limit as $\sigma\rightarrow 0$  in \eqref{lemprelkin1},
we finally get \eqref{lemmakinecform}.
\finedim

We next consider initial data of the form 
$$u_0(x)=v(x,\varphi(x))$$ with $\varphi(x)\in L^\infty(\R^N)$. By assumptions (f\ref{f1}) (f\ref{f2}), we have that   $v(x,\varphi(x))\in L^\infty(\R^N)$ if  $\varphi(x)\in L^\infty(\R^N)$.

\begin{lem}\label{uepestimlem} Assume (f\ref{f1}) (f\ref{f2}) and  let $u\in L^\infty(\R^+\times\R^N)\cap C([0,+\infty);L^1_{loc}(\R^N))$ be the weak solution of  
\eqref{porousequ}, \eqref{porousequinicond}  with    $u_0(x)=v(x,\varphi(x))$ for some function $\varphi\in  L^\infty(\R^N)$, and $v(x,p)$ defined as in 
\eqref{correctprel}.  
Assume in addition that 
\begin{equation}\label{varphigli1}v(x,\varphi(x))-v(x,0)\in L^1(\R^N).\end{equation}
Then the following holds:
\begin{itemize} 
\item[(i)] For all $t\geq 0$ and  $p\geq 0$
\beqs \int_{\R^N}(u(t,x)-v(x,p))_{+}dx\leq  \int_{\R^N}(v(x, \varphi(x))-v(x,p))_{+}dx< \infty;\eeqs
for all $t\geq 0$ and  $p\leq 0$ 
\beqs \int_{\R^N}(u(t,x)-v(x,p))_{-}dx\leq  \int_{\R^N}(v(x, \varphi(x))-v(x,p))_{-}dx< \infty.\eeqs
\item[(ii)] If $p_1<0<p_2\in \R$ are such that $p_1\leq \varphi(x)\leq p_2$ for any $x\in \R^N$, then
$$ v(x,p_1)\leq u(t,x)\leq v(x,p_2)\quad \text{for all }(t,x)\in\R^+\times\R^N.$$
\item[(iii)] Let $\eta$ be the positive measure on $\R$ defined by 
$$\int_{\R}\xi(p) d\eta(p):=\int_{\R^+\times\R^N\times\R}\xi(p)m(t,x,p)dtdx,$$ 
for  $\xi\in C_c(\R)$, then 
$$\eta\le \eta_0\quad\text{in }\mathcal{D}'(\R)$$ where 
$$\eta_0:={\bf 1}_{\{p>0\}}\|(v(x,\varphi(x))-v(x,p))_+\|_{L^1(\R^N)}+{\bf 1}_{\{p<0\}}\|(v(x,\varphi(x))-v(x,p))_-\|_{L^1(\R^N)}.$$
\end{itemize} 
\end{lem}
\proof 
For $p\ge0$, we have that 
\beq\label{l1inftyfinitvds} (v(x,\varphi(x))-v(x,p))_+\in L^1(\R^N).\eeq Indeed, by the monotonicity of $v(x,\cdot)$ and \eqref{varphigli1}, 
$$0\leq (v(x,\varphi(x))-v(x,p))_+=(v(x,\varphi(x))-v(x,p)){\bf 1}_{\{\varphi(x)>p\}}\leq (v(x,\varphi(x))-v(x,0))_+\in L^1(\R^N).$$
Similarly, one can prove that for $p\leq 0$, \beq\label{l1inftyfinitvdsbis}  (v(x,\varphi(x))-v(x,p))_-\in L^1(\R^N).\eeq
Then, we apply  (ii) of Theorem \ref{afsthm2} with  $u_1(t,x)=u(t,x)$,  $u_2(t,x)=v(x,p)$,  $u_{01}(x)=v(x,\varphi(x))$, $u_{02}(x)=v(x,p)$ with $p\ge0$.
By using that $$(u_{01}-u_{02})_+\in L^1(\R^N),$$
choosing $\phi(t,x)=\phi_k(t,x)$ with $\{\phi_k\}$ a sequence of functions in $C^\infty_c([0,+\infty)\times\R^N)$ approximating the function ${\bf 1}_{[0,t)\times\R^N}$ and letting $k\to+\infty$, 
we get the first inequality in (i). Similarly, the second inequality is obtained by
(ii) of Theorem \ref{afsthm2} applied to $u_1(t,x)=v(x,p)$ and $u_2(t,x)=u(t,x)$.

The monotonicity of $v(x,\cdot)$ implies that if $p_1<0<p_2$ are such that $p_1\leq \varphi(x)\leq p_2$, then 
$$v(x,p_1)\leq v(x,\varphi(x))\leq v(x,p_2).$$ Therefore,  (ii) is  a consequence of (i).

Finally, from \eqref{u-v+eq} and \eqref{l1inftyfinitvds}, we infer that for any $\xi\in C_c(\R)$
$$\int_0^{+\infty}dp\int_{\R^+\times\R^N}\xi(p)m(t,x,p)dtdx\leq \int_0^{+\infty}\xi(p)\|(v(x,\varphi(x))-v(x,p))_+\|_{L^1(\R^N)}dp.$$
Similarly,   from  \eqref{u-v+eq} and \eqref{l1inftyfinitvdsbis},
$$\int_{-\infty}^0dp\int_{\R^+\times\R^N}\xi(p)m(t,x,p)dtdx\leq \int_{-\infty}^0\xi(p)\|(v(x,\varphi(x))-v(x,p))_-\|_{L^1(\R^N)}dp.$$
Adding the two previous inequalities we get  (iii) and this concludes the proof of the lemma.
\finedim


The kinetic formulation for \eqref{porousequ}, \eqref{porousequinicond} is finally obtained by deriving \eqref{u-v+eq} with respect to $p$:
\begin{equation}\label{chieq}\begin{cases}\frac{\p}{\p t}\left(\frac{\p v}{\p p}(x,p) \chi_+\right)-\Delta  \chi_+=\frac{\p m}{\p p}(t,x,p)\\
m(t,x,p)=|\nabla f(x,u)|^2\delta_{f(x,u)}(p) \end{cases}\end{equation}
where
\beq\label{chiprelres} \chi_+(t,x,p):={\bf 1}_{\{v(x,p)<u(t,x)\}}.\eeq
Remark that since  $v(x,\cdot)=f(x,\cdot)^{-1}$ is monotone increasing, we have
\beqs \chi_+(t,x,p)={\bf 1}_{\{p<f(x,u(t,x))\}}.\eeqs
In the homogeneous case,  $v(x,p)=f^{-1}(p)$ does not depend on $x$. If we make the change of variable $v=f^{-1}(p)$, equation \eqref{chieq} becomes 
\eqref{paraconseqnoxkin}.

In order to make sense to the equation  \eqref{chieq}, we require that $g(x,p)=v(x,p)$ satisfies the following assumption: 
\resume{enumerate}
\item\label{f3}
$\frac{\partial g}{\partial p}\in L^{1}_{loc}(\R^N\times\R).$
\end{enumerate}
We are now ready to give the definition of kinetic solution for \eqref{porousequ}, \eqref{porousequinicond}.
\begin{de}\label{kinetciheterdef}
A kinetic solution of \eqref{poroueqhom}, \eqref{initialhom}  is a function $u$, with  $u(t,x)-v(x,0)\in L^\infty([0,\infty);L^1(\R^N))$, such that 
\begin{itemize}
\item[(i)]  $f(x,u(t,x))\in L_{loc}^2(\R^+; H^1_{loc}(\R^N))$;
\item[(ii)] $\chi_+(t,x,p)={\bf 1}_{\{v(x,p)<u(t,x)\}}$ satisfies  \eqref{chieq}    in the sense of distributions in $[0,+\infty)\times \R^N\times\R$, with initial datum 
$\chi(0,x,p)={\bf 1}_{\{v(x,p)<u_0(x)\}} $;
\item[(iii)] If $n$ is the positive measure on $\R$ defined by 
$$\int_{\R}\xi(p) dn(p):=\int_{\R^+\times\R^N\times\R}\xi(p)m(t,x,p)\,dt\,dx$$ for $\xi\in C_c(\R)$, 
then there exists $\eta\in L^\infty(\R)$ such that $\eta\rightarrow 0$ as $|p|\rightarrow\infty$ and
$$n\le \eta$$
in the sense of distributions in $\R$.
\end{itemize}
\end{de}
We conclude this section by showing that the weak solution of \eqref{porousequ}, \eqref{porousequinicond} with suitable initial condition is also kinetic solution. 
Precisely we have:
\begin{prop}\label{chieqeqprop} Assume (f\ref{f1}), (f2), (f\ref{f3}) and  let $u\in L^\infty(\R^+\times\R^N)\cap C([0,+\infty);L^1_{loc}(\R^N))$ be the weak solution of  
\eqref{porousequ}, \eqref{porousequinicond}  with    $u_0(x)=v(x,\varphi(x))$ for some function $\varphi\in  L^\infty(\R^N)$ such that 
$v(x,\varphi(x))-v(x,0)\in L^1(\R^N)$. Then
 the function $\chi_+(t,x,p)$ defined in \eqref{chiprelres} is a solution of \eqref{chieq}  in the sense of distributions in $[0,+\infty)\times\R^N\times\R$  with initial condition $\chi_+(0,x,p)={\bf 1}_{\{p<\varphi(x)\}}.$
 In particular there exists a kinetic solution of \eqref{porousequ}, \eqref{porousequinicond}.
 \end{prop}
\proof
Take $\phi=\phi(t,x)\in C^\infty_c([0,\infty)\times\R^N) $ and $\xi=\xi(p)\in C^\infty_c(\R)$, then
by \eqref{u-v+eq} \beqs\begin{split} <\frac{\p m}{\p p},\phi\xi>&=-<m, \phi\xi'>\\&=-\int_{\R^+\times\R^N}\,dt\,dx\,\phi_t(t,x)\int_\R( u(t,x)-v(x,p))_+ \xi'(p)\,dp\\&
-\int_{\R^N}\,dx\,\phi(0,x)\int_\R( u_0(x)-v(x,p))_+ \xi'(p)\,dp\\&
-\int_{\R^+\times\R^N}\,dt\,dx\,\Delta\phi(t,x)\int_\R(f(x,u(t,x))-p)_+\xi'(p) \,dp\\&
=-\int_{\R^+\times\R^N\times\R} \left(\phi_t(t,x)\frac{\p v}{\p p}(x,p)\chi_+(t,x,p)+ \Delta\phi(t,x)) \chi_+(t,x,p) \right)\xi(p)\,dt\,dx\,dp\\&
-\int_{\R^N\times\R}\phi(0,x)\frac{\p v}{\p p}(x,p){\bf 1}_{\{p<f(x,u_0(x))\}} \xi(p)\,dx\,dp,
\end{split}\eeqs  where the last equality is obtained by integrating by parts with respect to $p$ and using that
$$ \chi_+(t,x,p)={\bf 1}_{\{v(x,p)<u(t,x)\}}={\bf 1}_{\{p<f(x,u(t,x))\}}.$$
 Since  in addition $$f(x,u_0(x))=f(x,v(x,\varphi(x)))=\varphi(x),$$ and thus
 $${\bf 1}_{\{p<f(x,u_0(x))\}} ={\bf 1}_{\{p<\varphi(x)\}},$$
we conclude that  $\chi_+(t,x,p)$  is a solution of \eqref{chieq}  in the sense of distributions in $[0,+\infty)\times\R^N\times\R$  with initial condition $\chi_+(0,x,p)={\bf 1}_{\{p<\varphi(x)\}}.$

We have shown that the weak solution $u$ satisfies (i) and (ii) of Definition \ref{kinetciheterdef}. 
By  (i) of Lemma \ref{uepestimlem} and $v(x,\varphi(x))-v(x,0)\in L^1(\R^N)$, we also have that $u(t,x)-v(x,0)\in L^\infty([0,\infty);L^1(\R^N))$. Finally (iii) of Lemma \ref{uepestimlem}    implies that (iii) of   Definition \ref{kinetciheterdef} holds true. Thus  $u$ is kinetic solution of \eqref{porousequ}, \eqref{porousequinicond}.
\finedim


\section{Proof of Theorem \ref{mainthm} }\label{proofmain}
Let $u^\ep(t,x,\om)$ be the weak solution of \eqref{uepeq} whose existence is guaranteed by Theorem \ref{afsthm2}. Let us define
\beq\label{v}v(\om,p):=g(\om,p),\quad\text{where } g(\om,\cdot)=f^{-1}(\om,\cdot),\eeq
\beqs \chi^\ep_+(t,x,p,\om):={\bf 1}_{\{v\left(T\left(\frac{x}{\ep}\right)\om,p\right)<u^\ep(t,x,\om)\}}={\bf 1}_{\{p<f\left(T\left(\frac{x}{\ep}\right)\om,u^\ep(t,x,\om)\right)\}}\eeqs and 
\beqs \chi^\ep_-(t,x,p,\om):={\bf 1}_{\{v\left(T\left(\frac{x}{\ep}\right)\om,p\right)>u^\ep(t,x,\om)\}}={\bf 1}_{\{p>f\left(T\left(\frac{x}{\ep}\right)\om,u^\ep(t,x,\om)\right)\}}.\eeqs 
Then, by Proposition \ref{chieqeqprop}, for a.e.  $\om\in\Om$, $\chi^\ep_+$ and $\chi^\ep_-$ are respectively   solutions in the sense of distributions in $\R^+\times\R^N\times\R$ of 

\begin{equation}\label{chiepeq}\begin{cases}\frac{\p}{\p t}\left(\frac{\p v}{\p p}\left(T\left(\frac{x}{\ep}\right)\om,p\right) \chi_+^\ep\right)-\Delta  \chi_+^\ep=\frac{\p m^\ep}{\p p}(t,x,p,\om)\\
\chi^\ep_+(0,x,p)={\bf 1}_{\{p<\varphi(x)\}} 
\end{cases}
\end{equation}
\begin{equation*}\begin{cases}\frac{\p}{\p t}\left(\frac{\p v}{\p p}\left(T\left(\frac{x}{\ep}\right)\om,p\right) \chi_-^\ep\right)-\Delta  \chi_-^\ep=-\frac{\p m^\ep}{\p p}(t,x,p,\om)\\
 \chi^\ep_-(0,x,p)={\bf 1}_{\{p>\varphi(x)\}} 
 \end{cases}
 \end{equation*}
 where 
\beqs m^\ep(t,x,p,\om)=\left|\nabla \left[f\left(T\left(\frac{x}{\ep}\right)\om,u^\ep\right)\right]\right|^2\delta_{f\left(T\left(\frac{x}{\ep}\right)\om,u^\ep\right)} (p).\eeqs
Moreover, by Lemma \ref{uepestimlem} we have the following estimates:
\begin{itemize} 
\item If $p_1<0<p_2\in \R$ are such that $p_1\leq \varphi(x)\leq p_2$ for any $x\in \R^N$, then for all $(t,x)\in\R^+\times\R^N$ and  a.e. $\om\in\Om$
\beq\label{estim2} v\left(T\left(\frac{x}{\ep}\right)\om,p_1\right)\leq u^\ep(t,x,\om)\leq v\left(T\left(\frac{x}{\ep}\right)\om,p_2\right).\eeq
\item Let $n^\ep$ be the positive measure on $\R$ defined by 
$$\int_{\R}\xi(p)\, dn^\ep(p):=\int_{\R^+\times\R^N\times\R\times\Om}\xi(p)m^\ep(t,x,p,\om)\,dt\,dx\,d\mu$$ for $\xi\in C_c(\R)$, 
then 
\beq\label{estim3}n^\ep\le \eta\eeq in the sense of distributions in $\R$, where $\eta\in L^\infty(\R)$ with  compact support, is  defined by
\beq\label{estim4}\begin{split}\eta:=&{\bf 1}_{\{p>0\}}\int_{\R^N\times\Om}\left[v\left(T\left(\frac{x}{\ep}\right)\om,\varphi(x)\right)-v\left(T\left(\frac{x}{\ep}\right)\om,p\right)\right]_+\,dx\,d\mu\\&
+{\bf 1}_{\{p<0\}}\int_{\R^N\times\Om}\left[v\left(T\left(\frac{x}{\ep}\right)\om,\varphi(x)\right)-v\left(T\left(\frac{x}{\ep}\right)\om,p\right)\right]_-\,dx\,d\mu\\&
= {\bf 1}_{\{p>0\}}\int_{\R^N\times\Om}\left[v\left(\om,\varphi(x)\right)-v\left(\om,p\right)\right]_+\,dx\,d\mu\\&
+{\bf 1}_{\{p<0\}}\int_{\R^N\times\Om}\left[v\left(\om,\varphi(x)\right)-v\left(\om,p\right)\right]_-\,dx\,d\mu,
\end{split}
\eeq by using the invariance of $T$ with respect to $\mu$.
\end{itemize}
Remark that  the fact that $\eta$ has compact support  is a consequence of the monotonicity of $v(\om,\cdot)$ and the assumption $\varphi\in L^\infty(\R^N)$. 
\begin{lem}\label{mepconvlem}
There exists a subsequence of the measures $\{m^\ep\}$, still denoted by $\{m^\ep\}$ and a measure $m^0$ on $\R^+\times\R^N\times\R\times\Om$ such that $m^\ep$ converges to $m^0$ weakly in the sense of measures in $\R^+\times\R^N\times\R\times\Om$. Moreover, if $n^0$ is the positive measure on $\R$ defined by 
$$\int_{\R}\xi(p)\, dn^0(p):=\int_{\R^+\times\R^N\times\R\times\Om}\xi(p)m^0(t,x,p,\om)\,dt\,dx\,d\mu$$ for $\xi\in C_c(\R^N),$
then 
\beqs n^0\le \eta\eeqs in the sense of distributions in $\R$, where $\eta\in L^\infty(\R)$ with compact support. 

\end{lem}
\proof
The lemma is an immediate consequence of  \eqref{estim3} and \eqref{estim4}. 
\finedim


The functions $\{\chi^\ep_+\}_\ep$ and $\{\chi^\ep_-\}_\ep$ are obviously bounded uniformly in $\ep$, then by Proposition  \ref{propstocconv}, there exist $\chi^0_+,\,\chi^0_-\in L^\infty(\R^+\times\R^N\times\R\times\Om)$ such that, up to subsequence, $\{\chi^\ep_+\}$ and $\{\chi^\ep_-\}$ stochastically two-scale converge in the mean respectively to
 $\chi^0_+$ and $\chi^0_-$ as $\ep\rightarrow0$. 
 \begin{lem}\label{ergodiclimit}The functions $\chi^0_+$ and $\chi^0_-$ are independent of $\om\in\Om$.
\end{lem}
\dim
 In \eqref{chiepeq}, take as test function $\phi(t,x)\ep^2\psi\left(T\left(\frac{x}{\ep}\right)\om\right)\xi(p)$, where 
$\phi\in C^\infty_c(\R^+\times\R^N)$, $\xi\in  C^\infty_c(\R)$ and $\psi\in D^\infty(\Om)$,  $D^\infty(\Om)$ being the set defined in \eqref{Dinfinity}.
 Using that  for a.e. $\om\in\Om$, a.e. $x\in\R^N$, $$\frac{\p}{\p x_i}\left(\psi\left(T\left(\frac{x}{\ep}\right)\om\right)\right)=\frac{1}{\ep} D_i\psi\left(T\left(\frac{x}{\ep}\right)\om\right),$$
 we get 
\beqs\begin{split} &\int_{\R^+\times\R^N\times\R\times\Om}\chi_+^\ep\left[\ep^2 \psi\xi\frac{\p v}{\p p}\partial_t\phi+\ep^2\psi\xi\Delta_x\phi+2\ep\xi\nabla_x\phi D_\om\psi+\phi\xi\Delta_\om \psi\left(T\left(\frac{x}{\ep}\right)\om\right)\right]\,dt\,dx\,dp\,d\mu\\&
= \int_{\R^+\times\R^N\times\R\times\Om}\ep^2 \phi\psi m^\ep \xi'(p)\,dt\,dx\,dp\,d\mu.\end{split}\eeqs
By Lemma \ref{mepconvlem}, the right hand-side of the identity above goes to 0 as $\ep\rightarrow 0$. 
Hence, passing to the limit as $\ep\rightarrow 0$, we find that for almost every $(t,x,p)\in \R^+\times\R^N\times\R^N$, $\chi^0_+(t,x,p,\omega)$ is solution of 
$$\Delta_\om \chi^0_+=0$$ in the sense of distributions.
Then, Lemma \ref{liouvillelemma}  implies that $\chi^0_+$ is independent of $\om$. Similarly, we can prove that $\chi^0_-$  independent of $\om$.
\finedim

\begin{lem}\label{initailcondchi0} For a.e. $(x,p)\in\R^N\times\R$, $\chi^0_+(0,x,p)={\bf 1}_{\{p<\varphi(x)\}} $ and $\chi^0_-(0,x,p)={\bf 1}_{\{p>\varphi(x)\}} $.
\end{lem}
\begin{proof}
Since $\chi^\ep_+$ is solution of \eqref{chiepeq}, for any $\psi\in C_c^\infty([0,+\infty)\times\R^N\times\R)$,
\beqs\begin{split}
&\int_{\R^+\times\R^N\times\R\times\Om}\chi_+^\ep\left[\frac{\p v}{\p p}\left(T\left(\frac{x}{\ep}\right)\om,p\right)\partial_t\psi+\Delta_x\psi\right]\,dt\,dx\,dp\,d\mu\\&
\\&= -\int_{\R^N\times\R\times\Om}{\bf 1}_{\{p<\varphi(x)\}} \frac{\p v}{\p p}\left(T\left(\frac{x}{\ep}\right)\om,p\right)\psi(0,x,p)\,dx\,dp\,d\mu+\int_{\R^+\times\R^N\times\R\times\Om} m^\ep \partial_p\psi\,dt\,dx\,dp\,d\mu.
\end{split}
\eeqs
Passing to the limit as $\ep\to0$, by the two-scale convergence of $\chi^\ep_+$ to  $\chi^0_+$, Lemma \ref{mepconvlem} and the identity
$$\int_{\Om}\frac{\p v}{\p p}(\om,p)\,d\mu=\overline{g}'(p),$$
  we get
\beqs\begin{split}
&\int_{\R^+\times\R^N\times\R}\chi_+^0\left[\overline{g}'(p)\partial_t\psi+\Delta_x\psi\right]\,dt\,dx\,dp
\\&=- \int_{\R^N\times\R}{\bf 1}_{\{p<\varphi(x)\}}\overline{g}'(p)\ \psi(0,x,p)\,dx\,dp+\int_{\R^+\times\R^N\times\R\times\Om} m^0 \partial_p\psi\,dt\,dx\,dp\,d\mu.
\end{split}
\eeqs
Since $\overline{g}'(p)>0$ for a.e. $p\in\R$,  this implies  $\chi^0_+(0,x,p)={\bf 1}_{\{p<\varphi(x)\}}$.  Similarly, we get $\chi^0_-(0,x,p)={\bf 1}_{\{p>\varphi(x)\}} $.
\end{proof}
 Now, we want to identify $\chi^0_+$ and $\chi^0_-$. 
 Let us denote $$\overline{\chi}_+(t,x,p):={\bf 1}_{\{p<\overline{f}(\overline{u}(t,x))\}} $$ and 
 $$\overline{\chi}_-(t,x,p):={\bf 1}_{\{p>\overline{f}(\overline{u}(t,x))\}} $$
where $\overline{u}$ is the weak solution of \eqref{uequ}. Since $u_0(x,\om)=g(\om,\varphi(x))$,  recalling the definition \eqref{gbareq} of $\overline{g}(p)=\fs^{-1}(p)$, we see that $\overline{u}$  satisfies the initial condition
$$ \overline{u}(x,0)=\int_\Om g(\om,\varphi(x))\,d\mu=\overline{g}(\varphi(x)),$$ 
with $\varphi\in L^\infty(\R^N)$.
Moreover, by Assumption (H\ref{3}),
$$\overline{g}(\varphi(x))-\overline{g}(0)\in L^1(\R^N).$$
Thus, by  Proposition \ref{chieqeqprop}, we know that $\overline{\chi}_+$ and $\overline{\chi}_-$ are respectively solution in the sense of distribution in $[0,+\infty)\times\R^N\times\R$ of 
\begin{equation}\label{chibareq+}\begin{cases}\frac{\p}{\p t}\left(\overline{g}'(p)\overline{\chi}_+\right)-\Delta \overline{\chi}_+=\frac{\p \overline{m}}{\p p}(t,x,p)\\
\overline{\chi}_+(0,x,p)={\bf 1}_{\{p<\varphi(x)\}}
\end{cases}
 \end{equation}
\begin{equation}\label{chibareq-}\begin{cases}\frac{\p}{\p t}\left(\overline{g}'(p)\overline{\chi}_-\right)-\Delta \overline{\chi}_-=-\frac{\p \overline{m}}{\p p}(t,x,p)\\
\overline{\chi}_-(0,x,p)={\bf 1}_{\{p>\varphi(x)\}}
\end{cases}
\end{equation}
where 
\beqs \overline{m}(t,x,p)=\left|\nabla \left[\fs(\us)\right]\right|^2\delta_{\fs(\us)} (p).\eeqs

\begin{lem} \label{chi0+chibar-=0lem}We have 
$$\chi^0_+\overline{\chi}_-=0\text{ and } \chi^0_-\overline{\chi}_+=0\quad \text{for a.e. }(t,x,p)\in \R^+\times\R^N\times\R.$$
\end{lem}
Let us first give the  heuristic proof of the lemma.

{\bf Heuristic proof.}
We want to prove that 
\beq\label{euristic1}\frac{d}{dt}\int_{\R^N\times\R} \overline{g}'(p)\chi^0_+\overline{\chi}_-dxdp\leq0.\eeq Indeed, the previous inequality and 
$$\chi^0_+(0,x,p)\overline{\chi}_-(0,x,p)={\bf 1}_{\{p<\varphi(x)\}}{\bf 1}_{\{p>\varphi(x)\}}=0$$ imply that for $t>0$
$$\int_{\R^N\times\R} \overline{g}'(p)\chi^0_+\overline{\chi}_-\,dx\,dp\leq0.$$ Since $ \overline{g}'>0$, we infer that 
$\chi^0_+\overline{\chi}_-=0$ a.e. in $\R^+\times\R^N\times\R$.

Multiplying equation \eqref{chiepeq} by $\overline{\chi}_-$ and equation \eqref{chibareq-} by $\chi^\ep_+$,  and integrating by parts with respect to  $(x,p)$, we get respectively
\beqs \int_{\R^N\times\R}\left\{\frac{\p v}{\p p}\left(T\left(\frac{x}{\ep}\right)\om,p\right)\p_t \chi_+^\ep\overline{\chi}_-+\nabla  \chi_+^\ep\cdot\nabla\overline{\chi}_-\right\}\,dx\,dp
=-\int_{\R^N\times\R}m^\ep\p_p\overline{\chi}_-\,dx,\eeqs
\beqs \int_{\R^N\times\R}\left\{\overline{g}'(p)\p_t\overline{\chi}_-\chi_+^\ep+\nabla  \chi_+^\ep\cdot\nabla \overline{\chi}_-\right\}\,dx\,dp=\int_{\R^N\times\R} \overline{m}\p_p  \chi_+^\ep \,dx.\eeqs
Summing the two previous inequalities, we obtain 
 \beqs\begin{split}&\int_{\R^N\times\R}\left\{\frac{\p v}{\p p}\left(T\left(\frac{x}{\ep}\right)\om,p\right)\p_t \chi_+^\ep\overline{\chi}_-+\overline{g}'(p)\p_t\overline{\chi}_-\chi_+^\ep\right\}\,dx\,dp
 \\&=\int_{\R^N\times\R}\left\{-2\nabla  \chi_+^\ep\cdot\nabla \overline{\chi}_--m^\ep\p_p\overline{\chi}_-+\overline{m}\p_p  \chi_+^\ep\right\} \,dx\,dp
 \\&=\int_{\R^N\times\R}\left\{2\nabla \left[f\left(T\left(\frac{x}{\ep}\right)\om,u^\ep\right)\right]\cdot \nabla (\fs(\us))\right.\\&-
\left. \left|\nabla \left[f\left(T\left(\frac{x}{\ep}\right)\om,u^\ep\right)\right]\right|^2
 - \left|\nabla \left[\fs(\us)\right]\right|^2\right\}\delta_{f\left(T\left(\frac{x}{\ep}\right)\om,u^\ep\right)} (p)\delta_{\fs(\us)} (p)\,dx\\&
 \leq 0.
 \end{split}\eeqs
 Therefore
 \beqs\int_{\R^N\times\R}\left\{\frac{\p v}{\p p}\left(T\left(\frac{x}{\ep}\right)\om,p\right)\p_t \chi_+^\ep\overline{\chi}_-+\overline{g}'(p)\p_t\overline{\chi}_-\chi_+^\ep\right\}\,dx\,dp\leq0.\eeqs
 Integrating the previous inequality with respect to $\om$ and passing to the limit as $\ep\rightarrow 0$, using the fact that $\chi^0_+$ does not depend on $\om$,
 and that by assumption (H\ref{4})  $$\int_\Om  \frac{\p v}{\p p}(\om,p)\,d\mu=\overline{g}'(p),$$ we get
 $$\int_{\R^N\times\R} (\overline{g}'(p)\p_t \chi_+^0\overline{\chi}_-+\overline{g}'(p)\p_t\overline{\chi}_-\chi_+^0)\,dx\,dp\le0,$$
 which is  \eqref{euristic1}.

\dims {\bf of Lemma \ref{chi0+chibar-=0lem}} 
 
Let us prove that $\chi^0_+\overline{\chi}_-=0$ a.e. in $\R^+\times\R^N\times\R.$ We use the Kruzhkov's doubling variables method \cite{krus}.
 Let $0\le \phi\in C_c^\infty([0,\infty)\times\R^N)$ and let $\psi_m$, $\theta_n$ be classical smooth, compactly supported, approximations of identity in $\R^N$ and $\R$. For  $(t,x,s,y)\in(\R^+\times\R^N)^2$,  let us define
\beqs\Phi(t,x,s,y):=\phi\left(\frac{t+s}{2},\frac{x+y}{2}\right)\psi_m\left(\frac{x-y}{2}\right)\theta_n\left(\frac{t-s}{2}\right).\eeqs 
The functions $\psi_m$ and $\theta_n$ satisfy respectively, 
\beqs \int_{\R^{N}}g(y)\psi_m\left(\frac{x-y}{2}\right)\,dy\to g(x)\quad\text{as }m\to+\infty\eeqs
for a.e. $x$ and for any $g\in L^1(\R^N)$, 
\beqs \int_{\R}h(s)\theta_n\left(\frac{t-s}{2}\right)\,ds\to h(t),\quad\text{as }n\to+\infty\eeqs
for a.e. $t$ and for any $h\in L^1(\R).$
Let  $H_\sigma(s)$ be the  approximation of the Heaviside function given by \eqref{heaviside}.
Now  take $H_\sigma(p-\fs(\us(s,y)))\Phi(t,x,s,y)$ as test function in \eqref{chiepeq}, and integrate first in $(t,x,p)\in\R^+\times\R^N\times\R$ and then in $(s,y)\in\R^+\times\R^N$. 
Remark that even if $H_\sigma(p-\fs(\us(s,y)))$ does not have compact support,  by \eqref{estim2} $\chiep(t,x,p) H_\sigma(p-\fs(\us(s,y)))$ has compact support as a function of $p$, for $(t,x)$ and $(s,y)$ belonging to compact subsets of $[0,+\infty)\times\R^N$.  For simplicity of notation, in what follows, we will not write the domains of integration, and we will skip the dependence on $\om\in \Om$  denoting
$$v\left(\frac{x}{\ep},p\right):=v\left(T\left(\frac{x}{\ep}\right)\om,p\right),\quad f\left(\frac{x}{\ep},u^\ep\right):=f\left(T\left(\frac{x}{\ep}\right)\om,u^\ep\right).$$ 
 Then,  a.e. in $\Om$, we have 
\beq\label{primaeqlemm}\begin{split} &-\int\frac{\p v}{\p p}\left(\frac{x}{\ep},p\right)\chiep(t,x,p) H_\sigma(p-\fs(\us(s,y)))\partial_t\Phi\,dp\,dt\,dx\,ds\,dy\\
&-\int\frac{\p v}{\p p}\left(\frac{x}{\ep},p\right) {\bf 1}_{\{p<\varphi(x)\}}H_\sigma(p-\fs(\us(s,y)))\Phi(0,x,s,y)\,dp\,dx\,ds\,dy\\
&=\int \chiep(t,x,p) H_\sigma(p-\fs(\us(s,y)))\Delta_x\Phi dp\,dt\,dx\,ds\,dy\\&
-\int H_\s'\left(f\left(\frac{x}{\ep},u^\ep(t,x)\right)-\fs(\us(s,y))\right)\left|\nabla \left[f\left(\frac{x}{\ep},u^\ep(t,x)\right)\right]\right|^2\Phi \,dt\,dx\,ds\,dy.
\end{split}\eeq
It is easy to check that $$\Delta_x\Phi+\Delta_y\Phi+2\text{div}_y\nabla_x\Phi=\theta_n\psi_n\Delta\phi.$$
Hence, the first term in the  right-hand side of \eqref{primaeqlemm}, becomes
\beqs\begin{split} &\int \chiep(t,x,p) H_\sigma(p-\fs(\us(s,y)))\Delta_x\Phi \,dp\,dt\,dx\,ds\,dy\\&
=\int dp\,dt\,dx\,ds\, \chiep(t,x,p) \int H_\sigma(p-\fs(\us(s,y)))(-2\text{div}_y\nabla_x\Phi-\Delta_y\Phi+\theta_n\psi_n\Delta\phi)\, dy
\end{split}\eeqs
and by integrating by parts 
\beqs\begin{split} &
\int dp\,dt\,dx\,ds\, \chiep(t,x,p) \int H_\sigma(p-\fs(\us(s,y)))(-2\text{div}_y\nabla_x\Phi) \,dy\\&
=-\int dp\,dt\,dx\,ds\, \chiep(t,x,p) \int 2H'_\s(p-\fs(\us(s,y)))\nabla_y(\fs(\us(s,y)))\cdot \nabla_x\Phi \,dy\\&
=-\int \,dt\,dx\,ds\,dy 2\nabla_y(\fs(\us(s,y)))\cdot \nabla_x\Phi \int\chiep(t,x,p) H'_\s(p-\fs(\us(s,y)))\,dp\\&
=-\int  dtdsdy \int 2H_\s\left(f\left(\frac{x}{\ep},u^\ep(t,x)\right)-\fs(\us(s,y))\right)\nabla_y(\fs(\us(s,y)))\cdot \nabla_x\Phi \,dx\\&
=\int 2\text{div}_x\left[H_\s\left(f\left(\frac{x}{\ep},u^\ep(t,x)\right)-\fs(\us(s,y))\right)\nabla_y(\fs(\us(s,y)))\right]\Phi \,dt\,dx\,ds\,dy\\&
=\int 2H'_\s\left(f\left(\frac{x}{\ep},u^\ep(t,x)\right)-\fs(\us(s,y))\right)\nabla_x\left[f\left(\frac{x}{\ep},u^\ep(t,x)\right)\right]\cdot \nabla_y(\fs(\us(s,y)))\Phi \,dt\,dx\,ds\,dy.\\&
\end{split}\eeqs
We infer that
\beq\label{secondeqlemm}\begin{split} &\int \chiep(t,x,p) H_\sigma(p-\fs(\us(s,y)))\Delta_x\Phi \,dp\,dt\,dx\,ds\,dy\\&
=\int 2H'_\s\left(f\left(\frac{x}{\ep},u^\ep(t,x)\right)-\fs(\us(s,y))\right)\nabla_x\left[f\left(\frac{x}{\ep},u^\ep(t,x)\right)\right]\cdot \nabla_y(\fs(\us(s,y)))\Phi \,dt\,dx\,ds\,dy\\&
+\int \chiep(t,x,p) H_\sigma(p-\fs(\us(s,y)))(-\Delta_y\Phi +\theta_n\psi_n\Delta\phi) \,dp\,dt\,dx\,ds\,dy.\\&
\end{split}\eeq

From \eqref{primaeqlemm} and \eqref{secondeqlemm} we conclude that
\beq\label{xep+xibar-lemeq1}\begin{split} &-\int\frac{\p v}{\p p}\left(\frac{x}{\ep},p\right)\chiep(t,x,p) H_\sigma(p-\fs(\us(s,y)))\partial_t\Phi\,dp\,dt\,dx\,ds\,dy\\
&-\int\frac{\p v}{\p p}\left(\frac{x}{\ep},p\right) {\bf 1}_{\{p<\varphi(x)\}}H_\sigma(p-\fs(\us(s,y)))\Phi(0,x,s,y)\,dp\,dx\,ds\,dy\\&=
\int2H'_\s\left(f\left(\frac{x}{\ep},u^\ep(t,x)\right)-\fs(\us(s,y))\right)\nabla_x\left[f\left(\frac{x}{\ep},u^\ep(t,x)\right)\right]\cdot \nabla_y(\fs(\us(s,y)))\Phi \,dt\,dx\,ds\,dy\\&
+\int \chiep(t,x,p) H_\sigma(p-\fs(\us(s,y)))(-\Delta_y\Phi +\theta_n\psi_n\Delta\phi)  \,dp\,dt\,dx\,ds\,dy\\&
-\int H_\s'\left(f\left(\frac{x}{\ep},u^\ep(t,x)\right)-\fs(\us(s,y))\right)\left|\nabla_x \left[f\left(\frac{x}{\ep},u^\ep(t,x)\right)\right]\right|^2\Phi \,dt\,dx\,ds\,dy.
\end{split}\eeq

 Next, take $H_\sigma\left(f\left(\frac{x}{\ep},u^\ep(t,x)\right)-p\right)\Phi(t,x,s,y)$, as test function for the equation \eqref{chibareq-}  and integrate first in $(s,y,p)$ then in $(t,x)$. We have 
 \beq\label{xep+xibar-lemeq2}\begin{split} &-\int \overline{g}'(p)\ochi(s,y,p) H_\sigma\left(f\left(\frac{x}{\ep},u^\ep(t,x)\right)-p\right)\partial_s\Phi\,dp\,ds\,dy\,dt\,dx\\
&-\int\overline{g}'(p) {\bf 1}_{\{p>\varphi(y)\}}H_\sigma\left(f\left(\frac{x}{\ep},u^\ep(t,x)\right)-p\right)\Phi(t,x,0,y)\,dp\,dy\,dt\,dx\\=
&\int \ochi(s,y,p)H_\sigma\left(f\left(\frac{x}{\ep},u^\ep(t,x)\right)-p\right)\Delta_y\Phi \,dp\,ds\,dy\,dt\,dx\\&
-\int H_\s'\left(f\left(\frac{x}{\ep},u^\ep(t,x)\right)-\fs(\us(s,y))\right)|\nabla_y[\fs(\us(s,y))]|^2\Phi\, ds\,dy\,dt\,dx.
\end{split}\eeq
 
Summing \eqref{xep+xibar-lemeq1} and \eqref{xep+xibar-lemeq2}, we get 
\beqs\begin{split} &-\int\frac{\p v}{\p p}\left(\frac{x}{\ep},p\right)\chiep(t,x,p) H_\sigma(p-\fs(\us(s,y)))\partial_t\Phi\,dp\,dt\,dx\,ds\,dy\\&
-\int \overline{g}'(p)\ochi(s,y,p) H_\sigma\left(f\left(\frac{x}{\ep},u^\ep(t,x)\right)-p\right)\partial_s\Phi\,dp\,dt\,dx\,ds\,dy\\
&-\int\frac{\p v}{\p p}\left(\frac{x}{\ep},p\right) {\bf 1}_{\{p<\varphi(x)\}}H_\sigma(p-\fs(\us(s,y)))\Phi(0,x,s,y)\,dp\,dx\,ds\,dy\\
&-\int\overline{g}'(p) {\bf 1}_{\{p>\varphi(y)\}}H_\s\left(f\left(\frac{x}{\ep},u^\ep(t,x)\right)-p\right)\Phi(t,x,0,y)\,dp\,dt\,dx\,dy\\
&=\int \left\{\ochi(s,y,p)H_\sigma\left(f\left(\frac{x}{\ep},u^\ep(t,x)\right)-p\right)-\chiep(t,x,p) H_\sigma(p-\fs(\us(s,y)))\right\}\Delta_y\Phi \,dp\,dt\,dx\,ds\,dy\\&
-\int  H_\s'\left(f\left(\frac{x}{\ep},u^\ep(t,x)\right)-\fs(\us(s,y))\right)\left\{\left|\nabla_x \left[f\left(\frac{x}{\ep},u^\ep(t,x)\right)\right]\right|^2 +|\nabla_y[\fs(\us(s,y))]|^2 \right.\\&
\left.-2\nabla_x\left[f\left(\frac{x}{\ep},u^\ep(t,x)\right)\right]\cdot \nabla_y(\fs(\us(s,y)))\right\}\Phi \,dt\,dx\,ds\,dy\\&
+\int \chiep(t,x,p) H_\sigma(p-\fs(\us(s,y)))\theta_n\psi_n\Delta\phi \,dp\,dt\,dx\,ds\,dy\\&
\le \int \left\{\ochi(s,y,p)H_\sigma\left(f\left(\frac{x}{\ep},u^\ep(t,x)\right)-p\right)-\chiep(t,x,p) H_\sigma(p-\fs(\us(s,y)))\right\}\Delta_y\Phi \,dp\,dt\,dx\,ds\,dy\\&
+\int \chiep(t,x,p) H_\sigma(p-\fs(\us(s,y)))\theta_n\psi_n\Delta\phi \,dp\,dt\,dx\,ds\,dy.
\end{split}\eeqs
Then, letting $\delta$ go to 0, and integrating over $\Om$,  we get

\beq\label{mainthm3}\begin{split} &-\int\frac{\p v}{\p p}\left(T\left(\frac{x}{\ep}\right)\om,p\right)\chiep(t,x,p,\om) \ochi(s,y,p)\partial_t\Phi\,dp\,dt\,dx\,ds\,dy\,d\mu\\&
-\int \overline{g}'(p)\ochi(s,y,p) \chiep(t,x,p,\om)\partial_s\Phi\,dp\,dy\,ds\,dx\,dt\,d\mu\\
&\le \int\frac{\p v}{\p p}\left(T\left(\frac{x}{\ep}\right)\om,p\right) {\bf 1}_{\{p<\varphi(x)\}}\ochi(s,y,p)\Phi(0,x,s,y)\,dp\,dx\,ds\,dy\,d\mu\\&
+\int\overline{g}'(p) {\bf 1}_{\{p>\varphi(y)\}}\chiep(t,x,p,\om)\Phi(t,x,0,y)\,dp\,dt\,dx\,dy\,d\mu\\&
+\int \chiep(t,x,p,\om)  \ochi(s,y,p)\theta_n\psi_n\Delta\phi \,dp\,dt\,dx\,ds\,dy\,d\mu\\&
=: I_1+I_2+\int \chiep(t,x,p,\om)  \ochi(s,y,p)\theta_n\psi_n\Delta\phi \,dp\,dt\,dx\,ds\,dy\,d\mu.
\end{split}\eeq
Let us estimate the right hand-side of \eqref{mainthm3}.
 Recalling that $\ochi(s,y,p)={\bf 1}_{\{p>\overline{f}(\overline{u}(s,y))\}}$ and using \eqref{overlineginva}, we can estimate the first term in the right hand-side of \eqref{mainthm3} as follows
\beqs\begin{split} I_1&= \int\frac{\p v}{\p p}\left(T\left(\frac{x}{\ep}\right)\om,p\right) {\bf 1}_{\{p<\varphi(x)\}}\ochi(s,y,p)\Phi(0,x,s,y)\,dp\,dx\,ds\,dy\,d\mu\\&
=\int \,dp\,dx\,ds\,dy\,{\bf 1}_{\{p<\varphi(x)\}}\ochi(s,y,p)\Phi(0,x,s,y)\int \frac{\p v}{\p p}\left(T\left(\frac{x}{\ep}\right)\om,p\right) \,d\mu\\&
=\int \,dx\,ds\,dy\,\Phi(0,x,s,y)\int_{\overline{f}(\overline{u}(s,y))}^{\varphi(x)}\overline{g}'(p)\,dp\\&
=\int (\overline{g}(\varphi(x))-\overline{u}(s,y)))_+ \Phi(0,x,s,y)\,dx\,ds\,dy.
\end{split}\eeqs
Thus $I_1$ is actually independent of $\ep$. Recalling  that $\overline{u}(t,x)-\overline{g}(0)\in C([0,+\infty);L^1(\R^N))$ and $\overline{u}(0,x)=\overline{g}(\varphi(x))$, we see that letting  $n,m\to+\infty$, 
\beq\label{I1}I_1\to0.\eeq
Next, by letting first $\ep\to0$ and then $n,m\to+\infty$, by Lemma \ref{initailcondchi0}, we get 
\beq\label{I2}\begin{split}  
I_2\to \int\overline{g}'(p) {\bf 1}_{\{p>\varphi(x)\}}{\bf 1}_{\{p<\varphi(x)\}}\phi(0,x)\,dp\,dt\,dx=0.
\end{split}\eeq
Then, from \eqref{mainthm3}, we have

\beq\label{mainthm6}\begin{split} &-\int\frac{\p v}{\p p}\left(T\left(\frac{x}{\ep}\right)\om,p\right)\chiep(t,x,p,\om) \ochi(s,y,p)\psi_m\left(\frac{1}{2}\phi_t\theta_n+\frac{1}{2}\phi\theta_n'\right)\,dp\,dt\,dx\,ds\,dy\,d\mu\\&
-\int \overline{g}'(p)\ochi(s,y,p) \chiep(t,x,p,\om)\psi_m\left(\frac{1}{2}\phi_t\theta_n-\frac{1}{2}\phi\theta_n'\right)\,dp\,dt\,dx\,ds\,dy\,d\mu\\&
\le I_1+I_2
+\int \chiep(t,x,p,\om)  \ochi(s,y,p)\theta_n\psi_n\Delta\phi \,dp\,dt\,dx\,ds\,dy\,d\mu.
\end{split}\eeq
Letting first  $\ep\rightarrow0$ then $n,m\rightarrow\infty$, by the stochastically two-scale convergence of  $\chiep$ to $\chi^0_+$, using that $\chi^0_+$ is independent of $\om$ and recalling that 
\begin{equation*}\label{overlineginva}\int_\Om \frac{\p v}{\p p}\left(\om,p\right)\,d\mu=\overline{g}'(p),\eeqs
 we obtain
\beqs - \int \overline{g}'(p)\chi^0_+(t,x,p) \ochi(t,x,p)\p_t\phi(t,x)dpdtdx\leq C\int|\Delta\phi(t,x)|dt\,dx.\eeqs
Finally, taking $\phi(t,x)=\phi_1(t)\phi_2(x)$ with $\phi_1'<0$ in $[0,t]$, $\phi_2$ with compact support and converging to 1 in $C^2(\R^N)$,  we get that a.e., 
$$0\le -  \overline{g}'(p)\chi^0_+(t,x,p) \ochi(t,x,p)\leq0.$$
The previous inequalities and  $\overline{g}'>0$ imply that  $\chi^0_+ \ochi=0$  a.e.

 Similarly, we can prove that $\chi^0_-\overline{\chi}_+$ a.e., and this concludes the proof of the lemma.
\finedim

\begin{cor}\label{x0=xbarcor}We have
\beqs\chi^0_+=\overline{\chi}_+\quad\text{and}\quad \chi^0_-=\overline{\chi}_- \quad \text{for a.e. }(t,x,p)\in \R^+\times\R^N\times\R.\eeqs
\end{cor}
\proof
We know that a.e. $$0\leq\chi^0_+,\chi^0_-\le1,\quad \chi^0_++\chi^0_-=1,$$ and, by definition of $\overline{\chi}_+$ and $\overline{\chi}_-$
$$\overline{\chi}_+^2=\overline{\chi}_+,\quad \overline{\chi}_-^2=\overline{\chi}_-, \quad \overline{\chi}_++\overline{\chi}_-=1.$$
Using Lemma \ref{chi0+chibar-=0lem} and the previous properties, we get 
\beqs(\chi^0_+-\overline{\chi}_+)^2=(\chi^0_+)^2+(\overline{\chi}_+)^2-2\chi^0_+\overline{\chi}_+=(\chi^0_+)^2+(\overline{\chi}_+)^2-2\chi^0_+(1-\overline{\chi}_-)
\le \overline{\chi}_+-\chi^0_+.\eeqs
On the other hand 
\beqs(\chi^0_+-\overline{\chi}_+)^2=(\chi^0_+)^2+(\overline{\chi}_+)^2-(1-\chi^0_-)2\overline{\chi}_+\leq -( \overline{\chi}_+-\chi^0_+).\eeqs
Therefore we get
$(\chi^0_+-\overline{\chi}_+)^2=0$, i.e., $\chi^0_+=\overline{\chi}_+$. In the same way we can prove that $\chi^0_-=\overline{\chi}_-$.
\finedim

We are now ready to conclude the proof of Theorem \ref{mainthm}.  

Take $\phi=\phi(t,x)\in C^\infty_c(\R^+\times\R^N)$ and fix $p_0>0$.  Corollary \ref{x0=xbarcor} and the uniqueness of the kinetic solution of  the limit problem  \eqref{uequ} proven in \cite{cp}, imply that the whole sequence  $\chi^\ep_+(t,x,p,\om)$ two-scale converges to $\chi^0_+={\bf 1}_{\{p<\overline{f}(\overline{u}(t,x))\}}$. Therefore, 
$$\int_{\R^+\times\R^{N}\times\R\times\Om}{\bf 1}_{\{\overline{f}(\overline{u}(t,x))<p<p_0\}}\chi_+^\ep(t,x,p,\om)\frac{\p v}{\p p}\left(T\left(\frac{x}{\ep}\right)\om,p\right)\phi(t,x)dtdxdpd\mu\rightarrow 0$$ as $\ep\rightarrow 0$. The left-hand side above can be rewritten as follows:
\beqs\begin{split}&\int_{\R^+\times\R^{N}\times\R\times\Om}{\bf 1}_{\{\overline{f}(\overline{u}(t,x))<p<p_0\}}\chi_+^\ep(t,x,p,\om)\frac{\p v}{\p p}\left(T\left(\frac{x}{\ep}\right)\om,p\right)\phi(t,x)dt\,dx\,dp\,d\mu\\&
=\int_{\R^+\times\R^{N}\times\Om}dt\,dx\,d\mu\,\phi(t,x)\int_{\overline{f}(\overline{u})}^{\min\left[f\left(T\left(\frac{x}{\ep}\right)\om,u^\ep\right),p_0\right]}\frac{\p v}{\p p}\left(T\left(\frac{x}{\ep}\right)\om,p\right)dp
\\& =\int_{\R^+\times\R^{N}\times\Om}\left[\min\left[u^\ep(t,x,\om), v\left(T\left(\frac{x}{\ep}\right)\om,p_0\right)\right]-v\left({\small T\left(\frac{x}{\ep}\right)}\om,\overline{f}(\overline{u}(t,x))\right)\right]_+\phi(t,x)dtdxd\mu.
\end{split}\eeqs
Now, from \eqref{estim2} we infer that $\min\left[u^\ep(t,x,\om), v\left(T\left(\frac{x}{\ep}\right)\om,p_0\right)\right]=u^\ep(t,x,\om)$, provide $p_0>|\varphi|_\infty$. 
We deduce that, up to subsequence $$\int_{\R^+\times\R^{N}\times\Om}\left[u^\ep(t,x,\om)-v\left({\small T\left(\frac{x}{\ep}\right)}\om,\overline{f}(\overline{u}(t,x))\right)\right]_+\phi(t,x)dtdxd\mu\rightarrow 0$$ as $\ep\rightarrow 0$. 

Similarly, using the two-scale convergence of $\chi^\ep_-(t,x,p,\om)$ to $\chi^0_-={\bf 1}_{\{p>\overline{f}(\overline{u}(t,x))\}}$, we can prove that the previous limit with the positive part replaced by the negative one holds true. Hence, we get 
$$\int_{\R^+\times\R^{N}\times\Om}\left|u^\ep(t,x,\om)-v\left({\small T\left(\frac{x}{\ep}\right)}\om,\overline{f}(\overline{u}(t,x))\right)\right|\phi(t,x)dtdxd\mu\rightarrow 0$$ as $\ep\rightarrow 0$ and this proves \eqref{mainthmresult}.
 
 Now, let us show the weak star convergence of $\int_\Om u^\ep d\mu$ to $\overline{u}$. 
 For  any $\psi\in C_c(\R^+\times\R^N)$, we have 
  \beq\label{ultimathmmain}\begin{split}\lim_{\ep\rightarrow0}\int_{K}\int_\Om u^\ep(t,x,\om)\psi(t,x)d\mu dxdt&
 =\lim_{\ep\rightarrow0}\int_\Om\int_{K}v\left({\small T\left(\frac{x}{\ep}\right)}\om,\overline{f}(\overline{u}(t,x))\right)\psi(t,x)dtdxd\mu\\&
 =\int_Kdtdx\psi(t,x)\int_\Om v(\om,\overline{f}(\overline{u}(t,x))d\mu\\&
 =\int_K\overline{u}(t,x)\psi(t,x)dtdx.
 \end{split}\eeq

This concludes the proof of Theorem \ref{mainthm}.

\end{document}